\documentclass{amsart}
\usepackage{amssymb,graphicx}
\usepackage{mathrsfs}
\usepackage{color}
\usepackage{enumerate}
\usepackage{a4wide}
\usepackage[colorlinks=true,citecolor=blue,linkcolor=blue]{hyperref}

\newcommand{\Rl}{\mathbb{R}}
\newcommand{\Cplx}{\mathbb{C}}

\newcommand{\Ntrl}{\mathbb{N}}

\newcommand{\Dc}{\mathcal{D}}
\newcommand{\Ec}{\mathcal{E}}
\newcommand{\Fc}{\mathcal{F}}

\newcommand{\Hc}{\mathcal{H}}

\newcommand{\Op}{\mathrm{Op}}

\newcommand{\loc}{\mathrm{loc}}

\newcommand{\Tb}{\mathbb{T}}

\newcommand{\Pb}{\mathbb{P}}

\newcommand{\Ubb}{\mathbb{U}}

\newcommand{\Kb}{\boldsymbol{K}}

\def\XXint#1#2#3{{\setbox0=\hbox{$#1{#2#3}{\int}$ }
\vcenter{\hbox{$#2#3$ }}\kern-.6\wd0}}

\numberwithin{equation}{section}

\newtheorem{theorem}{Theorem}[section]
\newtheorem{proposition}[theorem]{Proposition}
\newtheorem{corollary}[theorem]{Corollary}
\newtheorem{definition}[theorem]{Definition}
\newtheorem{lemma}[theorem]{Lemma}

\newtheorem{assumption}{Assumption}

\theoremstyle{remark}
\newtheorem{example}[theorem]{Example}
\newtheorem{remark}[theorem]{Remark}

\title{$L_p$ estimates in the Androulidakis-Mohsen-Yuncken calculus}
\date{\today}

\author{E. McDonald}
\address{Laboratoire d’analyse et de math\'ematiques appliqu\'ees, Universit\'e Paris-Est Cr\'eteil, France}
\email{eamcd92@gmail.com}

\begin{document}
\maketitle{}

\begin{abstract}
 We prove that order zero operators in the pseudodifferential calculus associated to a filtration defined by Androulidakis, Mohsen and Yuncken are bounded on $L_p$ spaces for $1<p<\infty.$
\end{abstract}

\section{Introduction}
Recently, Androulidakis, Mohsen and Yuncken introduced a pseudodifferential calculus adapted to an arbitrary filtration of a manifold \cite{AMY2022arxiv}. Using this calculus, they resolved the long-standing Helffer-Nourrigat conjecture characterising maximally hypoelliptic operators in terms of representation theory.

We describe their setting as follows. Let $U\subset \Rl^d$ be open, and let $X_1,\ldots,X_n$ be smooth vector fields on $U$ satisfying the H\"ormander bracket condition on $U.$ That is, the Lie algebra of vector fields on $U$ generated by $\{X_1,\ldots,X_n\}$ spans the tangent space of $U$ at every point.
Equip each vector $X_j$ with a weight $w_j\geq 1.$ This data induces a \emph{filtered foliation} on $U,$ which is a filtration of the algebra of vector fields $\mathscr{X}(U)$ by locally finitely generated $C^\infty(U)$-submodules
\[
    0 = \Fc_0 < \Fc_1 \leq \cdots \leq \Fc_k = \mathscr{X}(U).
\]
Here, $\Fc_m$ is the $C^\infty(U)$-linear span of all the commutators of the form
\[
    [X_{j_1},[X_{j_2},\cdots [X_{j_{\ell-1}},X_{j_{\ell}}]\cdots]],\quad w_{j_1}+\cdots+w_{j_{\ell}}\leq m.
\]
Let $\Fc = \{\Fc_j\}_{j=0}^k.$ The $\Fc$-order of a differential operator $P$ is the least integer $m$ such that $P$ can be written in the form
\[
    P = \sum_{j_1,\ldots,j_k} M_{a_{j_1,\ldots,j_k}}X_{j_1}\cdots X_{j_k}
\]
where $M_{a_{j_1,\ldots,j_k}}$ is the operator of pointwise multiplication by a smooth function $a_{j_1,\ldots,j_k}$ on $U$ and 
\[
    \sum_{j_1,\ldots,j_k} w_{j_1}+\cdots w_{j_k} \leq m.
\]
Let $\mathrm{DO}_{\Fc}^m(U)$ denote the space of differential operators of $\Fc$-order at most $m.$

The algebra of differential operators on $U,$ $\mathrm{DO}(U) = \bigcup_{m\geq 0} \mathrm{DO}_{\Fc}^m(U),$ is a filtered algebra in the sense that
\[
    \mathrm{DO}^{m_1}_{\Fc}(U)\cdot \mathrm{DO}^{m_2}_{\Fc}(U) \subseteq \mathrm{DO}^{m_1+m_2}_{\Fc}(U),\quad m_1,m_2\geq 1.
\]

In \cite{AMY2022arxiv}, the authors constructed a filtered algebra $\Psi_{\Fc}(U) = \bigcup_{m\in \Rl} \Psi^m_{\Fc}(U)$ extending $\mathrm{DO}(U)$ which has most of the desirable properties of a pseudodifferential calculus. The construction was based on an ``approximate homogeneity" condition first described by Debord and Skandalis \cite{DebordSkandalis2014}, and later used as the basis of pseudodifferential calculus by van Erp and Yuncken \cite{vanErpYuncken2019}.

One of the desirable properties of a pseudodifferential calculus is that operators of order zero should act as bounded operators on $L_2(U),$ the space of square integrable functions, at least locally. This property for $\Psi^0_{\Fc}(U)$ is proved in \cite[Theorem 3.31(b)]{AMY2022arxiv}. The proof there followed the classic strategy of first showing that an operator with positive principal symbol has an approximate square root. Since symbols in this calculus are not functions but instead elements of some $C^*$-algebra, showing that the desired approximate square root belongs to the calculus is not straightforward. It is desirable to have a more self-contained argument. This is what we supply here. An off-shoot of this effort is that we also show that operators in $\Psi^0_{\Fc}(U)$ are also bounded on $L_p$ spaces for $1<p<\infty,$ in a suitable local sense.
\subsection{Acknowledgements}
The key idea in the proof of the main theorem, and arguably the main novelty of this paper, is the decomposition of kernels according to the $\Rl$-action described in the proof of Theorem \ref{AMY_zero_order_bounded}. The observation that this decomposition is almost orthogonal is due to Omar Mohsen, who kindly shared it with us. Thanks are also due to Rapha\"el Ponge for helpful comments and suggestions.

\section{Main results}

Let $X$ be a $C^\infty$-manifold, and let $\Fc$ be a filtration of the Lie algebra of vector fields on $X$ by $C^\infty$-submodules. Let $\Psi^m_{\Fc}(X)$ be the pseudodifferential operators of order $m\in \Cplx$ on $X,$ in the sense of Androulidakis, Mohsen and Yuncken. We will recall the definition of this calculus is recalled in Section \ref{AMY_section} below.

A pseudodifferential operator $T\in \Psi^m_{\Fc}(X)$ is associated to a one-parameter family $\{T_{\hbar}\}_{\hbar>0},$ with $T_1=T.$ This extension is unique up to the addition of a smoothly varying family of smoothing operators.

Given $1<p<\infty,$ the $L_p$-space $L_p(X)$ can be defined invariantly as a space of $\frac1p$-densities on $X.$ Alternatively, we can choose a fixed nonzero density on $X$ and identify $L_{p}(X)$ with a space of functions. Since our results are local, this distinction is not important.

Given $\phi\in C^\infty_c(X),$ we write $M_{\phi}$ for the operator of pointwise multiplication by $\phi$ on $L_p(X).$ We say that an operator $T$ is locally bounded on $L_p(X)$ if $M_{\phi}TM_{\psi}$ is bounded from $L_p(X)$ to $L_p(X)$ for all $\phi,\psi\in C^\infty_c(X).$
\begin{theorem}\label{main_theorem}
    Let $1<p<\infty,$ and let $T \in \Psi^m_{\Fc}(X).$ If $\Re(m)\leq 0,$ then $T$ is locally bounded on $L_p(X).$
    Moreover, the associated family $\{T_{\hbar}\}_{\hbar>0}$ is uniformly continuous as $\hbar\to 0$ in the sense that if $\phi,\psi\in C^\infty_c(X),$ then
    \[
        \limsup_{\hbar\to 0} \|M_{\phi}T_{\hbar}M_{\psi}\|_{L_p(X)\to L_p(X)} < \infty.
    \]
    If $\Re(m)<0,$ then the family $\{T_{\hbar}\}_{\hbar>0}$ is uniformly continuous as $\hbar\to 0$ on $L_p(X)$ for any $1\leq p\leq \infty,$ and moreover $T$ is locally compact in the sense that if $\phi,\psi\in C^\infty_c(X),$ then $M_{\phi}TM_{\psi}$ is compact on $L_p(X)$ for $1\leq p\leq \infty.$
\end{theorem}

\begin{remark}
    Concerning the proof, we make a few remarks.
\begin{itemize}
    \item{} With $p=2,$ there is no loss of generality in considering $m=0,$ since the closure of the calculus under composition and adjoint implies that if $\Re(m)=0,$ then $T^*T \in \Psi^0_{\Fc}(X).$ For $p\neq 2,$ the proof needs to allow general $m.$
    \item{} The proof is based on an almost-orthogonal decomposition, and the operator $T_{\hbar}$ is decomposed into a sum $\sum_{j=1}^\infty T_{\hbar}^{(j)}.$ The proof of Theorem \ref{main_theorem} shows that this sum is unconditional in $L_p.$
    \item{} In fact $M_{\phi}T_{\hbar}M_{\psi}$ in the theorem obeys a weak-type estimate, uniformly as $\hbar\to 0.$ The Marcinkiewicz interpolation theorem implies that the norm of $M_{\phi}T_{\hbar}M_{\psi}$ from $L_p$ to $L_p$ is bounded by a constant multiple of $\frac{p^2}{p-1}.$
\end{itemize}
\end{remark}
It should be noted that Theorem \ref{main_theorem} is not new. A very general class of operators was proved by Street to be bounded on $L_p$-spaces in \cite{Street-singular-integrals-2014}.
Those operators are defined as having symbol functions that obey standard estimates \cite[Definition 4.1.1]{Street-singular-integrals-2014}, and the operators in the Androulidakis-Mohsen-Yuncken calculus have kernels that obey similar estimates, this follows from \cite[Proposition 3.4]{AMY2022arxiv} by an argument similar to \cite[Corollary 45]{vanErpYuncken2019}. Therefore, $\Psi^0_{\Fc}(X)$ are order zero operators in Street's sense and  Theorem \ref{main_theorem} follows from \cite[Theorem 5.10.1]{Street-maximal-subellipticity-2023}. In fact, Street also proves the boundedness on $L_p$ spaces of so-called multiparameter pseudodifferential operators. Since these are typically not of weak-type $(1,1)$ (for example, consider a tensor product of two pseudodifferential operators), the proof of Theorem \ref{main_theorem} given here does not extend to that case.

Nevertheless, the main reason for this paper is to give a proof of Theorem \ref{main_theorem} within the context of the Androulidakis-Mohsen-Yuncken calculus. 

An important special case occurs when the filtration $\Fc$ arises as the sections of a filtration of $TX$ by vector bundles. This was the case considered in \cite{vanErpYuncken2019}. The boundedness of order zero operators on $L_p$ spaces in this case follows from Theorem \ref{main_theorem}, but in this case similar statements have long been known, see \cite[Page 103]{Goodman-lnm-1976}.

The proof of Theorem \ref{main_theorem} is mostly independent of the results in \cite{AMY2022arxiv}. The only results used from there are that the calculus is closed under composition and the adjoint. The $L_2$-boundedness of order zero operators was proved in \cite{AMY2022arxiv}, but here we obtain the same result by other methods.

Associated to the calculus $\Psi_{\Fc}(X)$ is a scale of Hilbert-Sobolev spaces adapted to the filtration $\Fc$ \cite[Section 3.5]{AMY2022arxiv}. Having the boundedness of operators of order zero $L_p$-spaces, we can develop a similar theory for $L_p$-Sobolev spaces. We supply some aspects of this theory in Section \ref{sobolev_section}, including a Sobolev embedding theorem and in Section \ref{elliptic_regularity_section} deduce some results concerning maximal subellipticity in $L_p$ for H\"ormander operators.

\section{Abstract kernel estimates}\label{abstract_section}
In this section and the next we will work in the setting of abstract harmonic analysis, similar to \cite{CoifmanWeiss1971,DengHan2009}.
The style of argument will be familiar to experts in harmonic analysis and similar arguments have appeared elsewhere, see in particular Goodman \cite{Goodman-lnm-1976}.

Let $(X,\rho,\mu)$ be a quasi-metric doubling measure space. That is,
there exists a constant $C_{\rho}$ such that
\[
    \rho(x,y)\leq C_{\rho}(\rho(x,y_0)+\rho(y,y_0)),\quad x,y,y_0\in X.
\]
and there exists a constant $C_{\mu}$ such that
\[
    \mu(B_X(x,2r)) \leq C_{\mu}\mu(B_X(x,r)),\quad r>0,\, x \in X
\]
where $B_X(x,r)$ denotes the ball about $x\in X$ of radius $r$ with respect to $\rho.$

Ultimately, $X$ will be a subset of Euclidean space, $\mu$ will be the Lebesgue measure and $\rho$ will be a quasi-metric defined by Nagel, Stein and Wainger \cite{NagelSteinWainger1985},
but for this section and the next we work in an abstract setting. 

We want to avoid making the assumption that there exist constants $C_X,Q$ such that
\[
    \int_{A<\rho(x,y)<B} \rho(x,y)^{-Q}\,d\mu(y) \leq C_X\log(B/A),\quad 0<A<B<\infty.
\]
This assumption was made in \cite{Goodman-lnm-1976}, but it is unclear whether it holds in our situation.

We will consider singular integral operators with kernels $K,$ given formally as infinite linear combinations $\sum_{j\in \Ntrl}\alpha_jK_j.$ We will require in different places various conditions on the individual terms $K_j,$ listed as Conditions \eqref{kernelsupport}-\eqref{kernelcancellation}, defined as follows.
\begin{definition}\label{kernelconditions}
    Let $\Kb = \{K_j\}_{j\in \Ntrl}$ be a sequence of locally integrable functions on $X\times X.$ We consider the following conditions on the sequence $\Kb$:
    \begin{enumerate}[{\rm (I)}]
        \item{}\label{kernelsupport} There is $C_1(\Kb)>0$ such that for all $j\in \mathbb{Z},$ $K_j(x,y)$ is supported in the set of $(x,y)\in X\times X$ where $\rho(x,y)\leq C_1(\Kb)2^{-j}.$
        \item{}\label{kernelmass} There is $C_{2}(\Kb)>0$ such that for all $j\in \Ntrl,$ we have
        \[
            \int_{X} |K_j(x,y)|\,d\mu(y)\,\leq C_{2}(\Kb).
        \]
        \item{}\label{kernelsmoothnessballvolume} There are constants $C_{3}(\Kb)$ and $M(\Kb)$ and a sequence of functions $\{I_j\}_{j\in \Ntrl}$ on $X\times X$ such that if $\rho(y,y_0)\leq M(\Kb)2^{-j}$ then
        \[
            |K_j(x,y)-K_{j}(x,y_0)| \leq 2^j\rho(y,y_0)I_j(x,y)
        \]
        and 
        \[
            \sup_{x\in X} \int_{X} |I_j(x,y)|\,d\mu(y),\, \sup_{y\in X} \int_X |I_j(x,y)|\,d\mu(x) \leq C_{3}(\Kb).
        \]
        \item{}\label{kernelcancellation} There is $C_4(\Kb)$ such that
        \[
            \left|\int_{X} K_j(x,y)\,d\mu(y)\right| \leq C_4(\Kb)2^{-j}.
        \]
    \end{enumerate}
    Given a finitely supported sequence $\{\alpha_j\}_{j\in \Ntrl}\subset \Cplx,$ denote
    \[
        \Kb_{\alpha}(x,y) = \sum_{j\in \Ntrl} \alpha_j K_j(x,y).
    \]
    Also let $\Kb^*$ be the sequence
    \[
        \Kb^* = \{(x,y)\mapsto \overline{K_j(y,x)}\}_{j\in \Ntrl}.
    \]
\end{definition}

Some initial remarks on the definition:
\begin{itemize}
    \item{} \eqref{kernelsupport} is indifferent to the order of $x$ and $y,$ so that if it applies to the functions $K_j(x,y)$ then it also applies to $\Kb^*$ with the same constant.
    \item{} All four conditions apply with the same constants if we replace $K_j(x,y)$ by $\overline{K_j(x,y)}$ or $\alpha_j K_j(x,y)$ where $|\alpha_j|\leq 1.$
    \item{} If $\Kb$ satisfies \eqref{kernelsupport}, then $K_j(x,y)=0$ for $2^j>C_1(\Kb)\rho(x,y)^{-1}$ and so the series defining $\Kb_{\alpha}(x,y)$ makes sense for arbitrary sequences $\alpha$ when $x\neq y.$
    \item{} The constant $M(\Kb)$ in Condition \eqref{kernelsmoothnessballvolume} is denoted differently to $C_1(\Kb),\ldots,C_4(\Kb)$ because the size of the constant has a different effect on the strength of the condition. Whereas a smaller value of $C_j(\Kb)$ for $j=1,\ldots,4$ implies a stronger condition, a smaller value of $M(\Kb)$ implies a weaker condition.
\end{itemize}

We start with the following relatively easy consequence of the first two conditions in Definition \ref{kernelconditions}
\begin{lemma}\label{uniform_parts}
    Let $\Kb = \{K_j\}_{j\in \Ntrl}$ satisfy conditions \eqref{kernelsupport} and \eqref{kernelmass} of Definition \ref{kernelconditions}. Then
    \begin{equation}\label{uniformsupport}
        \sup_{x\in X} \int_X \rho(x,y)|K_j(x,y)|\,d\mu(y) \leq C_1(\Kb)C_{2}(\Kb)2^{-j},\quad j\in \Ntrl.
    \end{equation}
\end{lemma}
\begin{proof}
    For $x,y\in X,$ we have
    \[
        \rho(x,y)\chi_{B_X(x,C_1(\Kb)2^{-j})}(y)\leq C_1(\Kb)2^{-j}\chi_{B_X(x,C_1(\Kb)2^{-j})}(y).
    \]
    and hence
    \[
        \rho(x,y)|K_j(x,y)| \leq C_1(\Kb)2^{-j}|K_{j}(x,y)|.
    \]
    Integrating over $y$ and using \eqref{kernelmass} delivers \eqref{uniformsupport}.
\end{proof}

The following shows that the sums $\Kb_{\alpha}$ from Definition \ref{kernelconditions} obey the estimates required of a Calder\`on-Zygmund kernel as in \cite[Chapter III, Theorem (2.4), p.74]{CoifmanWeiss1971}.
\begin{lemma}\label{coifman_weiss_condition}
    If $\Kb = \{K_j\}_{j\in \Ntrl}$ satisfies \eqref{kernelsupport} and \eqref{kernelsmoothnessballvolume} of Definition \ref{kernelconditions}, and $\Kb^*$ satisfies \eqref{kernelmass}, then for all bounded sequences $\alpha$ and $y\neq y_0\in X$ we have
    \[
        \int_{\rho(x,y)>2C_{\rho}\rho(y,y_0)} |\Kb_{\alpha}(x,y)-\Kb_{\alpha}(x,y_0)| \,d\mu(x) \leq C(\Kb)\sup_j \{|\alpha_j|\}
    \]
    where $C(\Kb)$ depends only on $C_1(\Kb)/M(\Kb)$, on $C_2(\Kb^*)$ and on $C_1(\Kb)C_3(\Kb).$
\end{lemma}
\begin{proof}
    Note that in region of integration, $\Kb_{\alpha}(x,y)$ is only evaluated when $x\neq y.$ Since $K_j(x,y)$ is non-zero only when $\rho(x,y)<C_1(\Kb)2^{-j},$ the sum
    \[
        \Kb_{\alpha}(x,y) = \sum_{j\in \Ntrl} \alpha_j K_j(x,y)
    \]
    is finite for every $x\neq y.$

    By definition,
    \begin{align*}
        &\int_{\rho(x,y)>2C_{\rho}\rho(y,y_0)} |\Kb_{\alpha}(x,y)-\Kb_{\alpha}(x,y_0)|\,d\mu(x)\\
        &\quad\leq \int_{\rho(x,y)>2C_{\rho}\rho(y,y_0)} \sum_{j\in \Ntrl} |\alpha_j||K_j(x,y)-K_j(x,y_0)|\,d\mu(x).
    \end{align*}
    We split the sum into parts where $\rho(y,y_0)< C_1(\Kb)2^{-j}$ and where $\rho(y,y_0)>C_1(\Kb)2^{-j}.$
    In the region of integration, we have $\rho(x,y)>2C_{\rho}\rho(y,y_0),$ and hence by the quasi-triangle inequality
    \[
        C_\rho(\rho(x,y_0)+\rho(y,y_0)) \geq \rho(x,y) > 2C_{\rho}\rho(y,y_0).
    \]
    Therefore $\rho(x,y_0)> \rho(y,y_0).$ If $\rho(y,y_0)>C_1(\Kb)2^{-j},$ then we have both $\rho(x,y_0)$ and $\rho(x,y)$ greater than $C_1(\Kb)2^{-j}.$ Hence,
    \[
        \sum_{\rho(y,y_0)>C_1(\Kb)2^{-j}} |\alpha_j||K_j(x,y)-K_j(x,y_0)| = 0.
    \]
    Thus
    \begin{align*}
        &\int_{\rho(x,y)>2C_{\rho}\rho(y,y_0)} |\Kb_{\alpha}(x,y)-\Kb_{\alpha}(x,y_0)|\,d\mu(x)\\
        &= \sum_{2^j\leq C_1(\Kb)\rho(y,y_0)^{-1}} \int_{\rho(x,y)>2C_\rho\rho(y,y_0)} |\alpha_j||K_j(x,y)-K_j(x,y_0)|\,d\mu(x)\\
        &\leq \sum_{2^{j}< M(\Kb)\rho(y,y_0)^{-1}} \int_{\rho(x,y)>2C_\rho\rho(y,y_0)} |\alpha_j||K_j(x,y)-K_j(x,y_0)|\,d\mu(x)\\
        &\quad + \sum_{ M(\Kb)\rho(y,y_0)^{-1} \leq 2^j \leq C_1(\Kb)\rho(y,y_0)^{-1}} \int_{\rho(x,y)>2C_\rho\rho(y,y_0)} |\alpha_j||K_j(x,y)-K_j(x,y_0)|\,d\mu(x).
    \end{align*}
    Applying Condition \eqref{kernelsmoothnessballvolume} to the first summand and Condition \eqref{kernelmass} to the second summand gives
    \begin{align*}
        &\int_{\rho(x,y)>2C_{\rho}\rho(y,y_0)} |\Kb_{\alpha}(x,y)-\Kb_{\alpha}(x,y_0)|\,d\mu(x)\\
        &\leq \sup_j\{|\alpha_j|\}\sum_{2^j < M(\Kb)\rho(y,y_0)^{-1}} 2^j\rho(y,y_0)\int_{X} I_j(x,y)\,d\mu(x)\\
        &\quad + \sup_{j}\{|\alpha_j|\} \sum_{M(\Kb)\rho(y,y_0)^{-1}\leq 2^j\leq C_1(\Kb)\rho(y,y_0)^{-1}}2C_2(\Kb^*)\\
        &\leq \sup_j\{|\alpha_j|\}C_3(\Kb) \sum_{2^j\leq C_1(\Kb)\rho(y,y_0)^{-1}} 2^j \rho(y,y_0)\\
        &\quad+ \sup_{j}\{|\alpha_j|\} \sum_{M(\Kb)\rho(y,y_0)^{-1}\leq 2^j\leq C_1(\Kb)\rho(y,y_0)^{-1}}2C_2(\Kb^*)
    \end{align*}
    By the numerical inequalities
    \[
        \sum_{2^j\leq x} 2^j \leq 2x,\quad x\geq 0
    \]
    and 
    \[
        |\{j\geq 0\;:\;A\leq 2^j \leq B\}| \leq 2+|\log(B/A)/\log(2)|
    \]
    we conclude
    \[
        \int_{\rho(x,y)>2C_{\rho}\rho(y,y_0)} |\Kb_{\alpha}(x,y)-\Kb_{\alpha}(x,y_0)|\,d\mu(x) \leq C(\Kb)\sup_j\{|\alpha_j|\}
    \]
    where $C(\Kb)$ depends only on $C_2(\Kb^*),$ on $C_1(\Kb)/M(\Kb)$ and on $C_1(\Kb)C_3(\Kb).$

\end{proof}

In order to obtain $L_2$ boundedness for $\Kb_{\alpha},$ we need to verify an almost-orthogonality condition on the kernels $\{K_j\}_{j\in \Ntrl}$ and also their adjoints.

\begin{theorem}\label{almostorthogonality}
    Let $\Kb = \{K_j\}_{j\in \Ntrl}$ and $\Kb^* = \{\overline{K_j(y,x)}\}_{j\in \Ntrl}$ both satisfy all of the conditions of Definition \ref{kernelconditions}. Assume without loss of generality that $C_j(\Kb^*)= C_j(\Kb)$ for $j=1,2,3,4$ and $M(\Kb^*)=M(\Kb).$ Then
    \[
        \sup_{x\in X} \int_X \left|\int_X K_j(x,y_0)\overline{K_{\ell}(y,y_0)}d\mu(y_0)\right|\,d\mu(y) \leq C(\Kb)2^{-|j-\ell|}
    \]
    where $C(\Kb)$ depends only on $C_2(\Kb),$ on $C_1(\Kb)C_3(\Kb),$ on $C_2(\Kb)C_4(\Kb)$ and on $C_1(\Kb)/M(\Kb).$

\end{theorem}
\begin{proof}
    Let
    \[
        b_{j,\ell}(x,y) = \int_{X} K_j(x,y_0)\overline{K_{\ell}(y,y_0)}\,d\mu(y_0).
    \]
    Note that
    \begin{equation}\label{l_close_to_j}
        \int_{X} |b_{j,\ell}(x,y)| \,d\mu(y) \leq C_2(\Kb)^2.
    \end{equation}
    We will show separately that
    \begin{equation}\label{l_lt_j}
        \sup_{x\in X}\int_X|b_{j,\ell}(x,y)|\,d\mu(y) \leq  C_2(\Kb)(C_1(\Kb)C_3(\Kb)+C_4(\Kb))  2^{\ell-j},\quad 2^{\ell-j}\frac{C_1(\Kb)}{M(\Kb)}\leq 1
    \end{equation}
    and
    \begin{equation}\label{j_lt_l}
        \sup_{x\in X}\int_X|b_{j,\ell}(x,y)|\,d\mu(y) \leq C_2(\Kb)(C_1(\Kb)C_3(\Kb)+C_4(\Kb))  2^{j-\ell},\quad 2^{j-\ell}\frac{C_1(\Kb)}{M(\Kb)}\leq 1.
    \end{equation}
    The combination of \eqref{l_close_to_j}, \eqref{l_lt_j} and \eqref{j_lt_l} implies the result.

    First, to show \eqref{l_lt_j}, we write $b_{j,\ell}$ as a sum as follows:
    \[
        b_{j,\ell}(x,y) = \int_X K_j(x,y_0)(\overline{K_{\ell}(y,y_0)}-\overline{K_{\ell}(y,x)})\,d\mu(y_0) + \overline{K_{\ell}(y,x)}\int_X K_{j}(x,y_0)\,d\mu(y_0).
    \]
    Exchanging the order of integration in the first term gives
    \begin{align*}
        \int_{X} |b_{j,\ell}(x,y)|\,d\mu(y) &\leq \int_X |K_j(x,y_0)|\left(\int_X\,|K_{\ell}(y,y_0)-K_{\ell}(y,x)|d\mu(y)\right)\,d\mu(y_0)\\
                                            &\quad + \int_{X} |K_{\ell}(y,x)|\,d\mu(y)\left|\int_X K_j(x,y_0)\,d\mu(y_0)\right|.
    \end{align*}
    Applying \eqref{kernelmass} for $\Kb^*$ and \eqref{kernelcancellation} for $\Kb$ to the second summand gives
    \begin{align*}
        \int_X |b_{j,\ell}(x,y)|\,d\mu(y) &\leq \int_X |K_j(x,y_0)|\left(\int_X\,|K_{\ell}(y,y_0)-K_{\ell}(y,x)|d\mu(y)\right)\,d\mu(y_0)\\
                                          &\quad+ C_{2}(\Kb)C_4(\Kb)2^{-j}.
    \end{align*}
    Since $\ell\geq 0,$ it follows that
    \begin{align*}
        \int_X |b_{j,\ell}(x,y)|\,d\mu(y) &\leq \int_X |K_j(x,y_0)|\left(\int_X\,|K_{\ell}(y,y_0)-K_{\ell}(y,x)|d\mu(y)\right)\,d\mu(y_0)\\
                                          &\quad + C_{2}(\Kb)C_4(\Kb)2^{\ell-j}.
    \end{align*}
    We therefore concentrate on bounding the first summand above.
    Condition \eqref{kernelsmoothnessballvolume} implies that
    \begin{equation}\label{integrated_kernel_smoothness}
        \int_{X}|K_{\ell}(y,y_0)-K_{\ell}(y,x)|d\mu(y) \leq C_{3}(\Kb)2^{\ell}\rho(x,y_0),\quad \rho(y_0,x)\leq M(\Kb)2^{-\ell}.
    \end{equation}

    By \eqref{kernelsupport}, in the support of $K_j(x,y_0)$ we have $\rho(x,y_0)\leq C_1(\Kb)2^{-j}.$ For \eqref{l_lt_j}, we assume that
    \[
        C_1(\Kb)2^{-j}\leq M(\Kb)2^{-\ell}
    \]
    which implies, via \eqref{integrated_kernel_smoothness}, that
    \begin{align*}
        &\int_X |K_j(x,y_0)|\left(\int_X\,|K_{\ell}(y,y_0)-K_{\ell}(y,x)|d\mu(y)\right)\,d\mu(y_0)\\
                                            &\leq C_{3}(\Kb)2^{\ell} \int_X \rho(x,y_0)|K_j(x,y_0)|\,d\mu(y_0).
    \end{align*}
    Applying \eqref{uniformsupport} to $\Kb$ delivers \eqref{l_lt_j}.

    Now we prove \eqref{j_lt_l}. Write
    \[
        b_{j,\ell}(x,y) = \int_{X} (K_j(x,y_0)-K_j(x,y))\overline{K_{\ell}(y,y_0)}\,d\mu(y_0) + K_j(x,y)\int_{X} \overline{K_{\ell}(y,y_0)}\,d\mu(y_0).
    \]
    Exchanging the order of integration we obtain
    \begin{align*}
        \int_{X} |b_{j,\ell}(x,y)|\,d\mu(y) &\leq \int_X \int_X |K_{\ell}(y,y_0)||K_j(x,y_0)-K_j(x,y)|\,d\mu(y)\,d\mu(y_0)\\
                                            &\quad + \int_X |K_j(x,y)|\,d\mu(y)\left|\int_X K_{\ell}(y,y_0)\,d\mu(y_0)\right|.
    \end{align*}
    By \eqref{kernelmass} and \eqref{kernelcancellation} applied to $\Kb,$ the second term is bounded above by $C_{2}(\Kb)C_4(\Kb)2^{-\ell}.$ We therefore concentrate on the first summand.

        Since $C_1(\Kb)2^{-\ell}\leq M(\Kb)2^{-j},$ within the support of $K_{\ell}(y,y_0),$ we have $\rho(y,y_0)\leq C_1(\Kb)2^{-j}.$  Applying \eqref{kernelsmoothnessballvolume} to $K_j$ and \eqref{uniformsupport} to $K_\ell$ we obtain
        \begin{align*}
            &\int_X \int_X |K_{\ell}(y,y_0)||K_j(x,y_0)-K_{j}(x,y)|\,d\mu(y)d\mu(y_0)\\
            &\leq 2^{\ell}\int_X \int_X \rho(y,y_0)|K_{\ell}(y,y_0)|I_j(x,y)\,d\mu(y)d\mu(y_0)\\
            &\leq C_1(\Kb)C_2(\Kb)2^{\ell-j}\int_{X} I_j(x,y)\,d\mu(y)\\
            &\leq C_1(\Kb)C_2(\Kb)C_{3}(\Kb)2^{\ell-j}.
        \end{align*}
        This completes the proof of \eqref{l_lt_j}.
    
\end{proof}

\section{Abstract Operator estimates}\label{abstract_operator_section}
    In this section we continue in the same setting as Section \ref{abstract_section}: $(X,\rho)$ is a quasi-metric space, and $\mu$ is a doubling measure on $X$ with respect to $\rho.$

    Let $\Kb = \{K_j\}_{j\in \Ntrl}$ be a kernel satisfying all of the conditions of Definition \ref{kernelconditions}, and assume that the same holds for $\Kb^*.$ Without loss of generality, we assume throughout that $C_j(\Kb^*)=C_j(\Kb)$ for $j=1,\ldots,4$
    and $M(\Kb^*)=M(\Kb).$

    Given a locally integrable function $K$ on $X\times X,$ denote $\Op(K)$ for the integral operator with kernel $K.$ That is,
    \[
        \Op(K)u(x) = \int_{X} K(x,y)u(y)\,d\mu(y),\quad u\in L_{\infty,\mathrm{comp}}(X).\; x\in X.
    \]

    \begin{theorem}\label{abstract_L_2_boundedness}
        Let $\Kb$ and $\Kb^*$ both satisfy all of the conditions of Definition \ref{kernelconditions}.
        For any finitely supported sequence $\{\alpha_j\}_{j\in \Ntrl},$ the operator $\Op(\Kb_{\alpha})$ is bounded on $L_2(X,\mu).$ Specficially, we have
        \[
            \|\Op(\Kb_{\alpha})\|_{L_2(X,\mu)\to L_2(X,\mu)} \leq 2^{\frac12}C(\Kb)^{\frac12}\max_j\{|\alpha_j|\}.
        \]
        where $C(\Kb)$ is the same constant as in Theorem \ref{almostorthogonality}. 
    \end{theorem}
    \begin{proof}
        Assume without loss of generality that $\max_j\{|\alpha_j|\} \leq 1.$
        We have
        \[
            \Op(\Kb_{\alpha}) = \sum_{j\in \Ntrl} \alpha_j\Op(K_j).
        \]
        Note that
        \[
            \|\Op(K_j)\Op(K_{\ell})^*\|_{L_{\infty}(X)\to L_{\infty}(X)} = \sup_{x\in X} \int_X \left| \int_X K_j(x,y_0)\overline{K_{\ell}(y,y_0)}\,d\mu(y_0)\right|\,d\mu(y).
        \]
        By Theorem \ref{almostorthogonality},
        \[
            |\alpha_j\alpha_\ell|\|\Op(K_j)\Op(K_\ell)^*\|_{L_{\infty}(X)\to L_{\infty}(X)} \leq C(\Kb)2^{-|j-\ell|},\quad j,\ell\in \Ntrl.
        \]
        Dually, we have
        \[
            |\alpha_j\alpha_\ell|\|\Op(K_j)\Op(K_\ell)^*\|_{L_{1}(X,\mu)\to L_{1}(X,\mu)} \leq C(\Kb)2^{-|j-\ell|},\quad j,\ell\in \Ntrl.
        \]
        By complex interpolation, it follows that
        \[
            |\alpha_j\alpha_\ell|\|\Op(K_j)\Op(K_\ell)^*\|_{L_{2}(X,\mu)\to L_{2}(X,\mu)} \leq C(\Kb)2^{-|j-\ell|},\quad j,\ell\in \Ntrl.
        \]
        Exchanging the roles of $\Kb$ and $\Kb^*$ gives
        \[
            |\alpha_j\alpha_\ell|\|\Op(K_j)^*\Op(K_\ell)\|_{L_{2}(X,\mu)\to L_{2}(X,\mu)} \leq C(\Kb)2^{-|j-\ell|},\quad j,\ell\in \Ntrl.
        \]
        That is,
        \begin{align*}
            \|\Op(\alpha_jK_j)\Op(\alpha_\ell K_\ell)^*\|_{L_{2}(X,\mu)\to L_{2}(X,\mu)}&+\|\Op(\alpha_j K_j)^*\Op(\alpha_{\ell}K_\ell)\|_{L_{2}(X,\mu)\to L_{2}(X,\mu)}\\
                                                                                        &\leq 2C(\Kb)2^{-|j-\ell|}.
        \end{align*}

        The $L_2$-boundedness of the sum $\Kb_{\alpha}$ follows from this estimate and Cotlar-Knapp-Stein almost-orthogonality.
    \end{proof}

    \begin{corollary}\label{weak_type_corollary}
        Let $\Kb$ and $\Kb^*$ satisfy Conditions \eqref{kernelsupport}, \eqref{kernelmass} and \eqref{kernelsmoothnessballvolume}. Definition \ref{kernelconditions}, with the same constants.        
        Then for any finitely supported sequence $\{\alpha_j\}_{j\in \Ntrl},$ if $\Op(\Kb_{\alpha})$ is bounded on $L_2(X),$ then $\Op(\Kb_{\alpha})$ has weak-type $(1,1)$ with norm depending only on $C_1(\Kb)C_3(\Kb),$ on $C_2(\Kb),$ on $C_1(\Kb)/M(\Kb)$
        on $\|\Op(\Kb_{\alpha})\|_{L_2(X)\to L_2(X)}$ and on $\sup_j|\alpha_j|.$
    \end{corollary}
    \begin{proof}
        Lemma \ref{coifman_weiss_condition} and Theorem \ref{abstract_L_2_boundedness} show that the kernel $\Kb_{\alpha}$ satisfies all of the conditions of \cite[Chapter III, Theorem (2.4), p.74]{CoifmanWeiss1971}.
        The weak-type estimate follows from that theorem.
    \end{proof}

    \begin{corollary}\label{abstract_L_p_boundedness}
        Let $\Kb$ and $\Kb^*$ satisfy all of the conditions of Definition \ref{kernelconditions} and let $\{\alpha_j\}_{j\in \Ntrl}\subset \Cplx$ be a finitely supported sequence. Then $\Op(\Kb_{\alpha})$ is bounded on $L_p(X,\mu)$ for all $1<p<\infty,$ and there exists a constant $C(\Kb,\max_j|\alpha_j|)$
        depending only on the quantities
        \begin{align*}
            C_1(\Kb)C_3(\Kb), C_2(\Kb), C_2(\Kb)C_4(\Kb), C_1(\Kb)M(\Kb)^{-1}\text{ and }\sup_j|\alpha_j|.
        \end{align*}
        Moreover,
        \[
            \|\Op(\Kb_{\alpha})\|_{L_p(X,\mu)\to L_p(X,\mu)} \leq C(\Kb,\max_j|\alpha_j|)\frac{p^2}{p-1},\quad 1<p<\infty.
        \]
    \end{corollary}
    \begin{proof}
        Combining Theorem \ref{abstract_L_2_boundedness} and Corollary \ref{weak_type_corollary}, we get the result for $1<p<2$ by Marcinkiewicz interpolation. Since the assumptions are symmetric in $\Kb$ and $\Kb^*,$ we get the same for the adjoint and hence by duality we have the bound for $2<p<\infty.$
    \end{proof}

    It is also worth mentioning the following much easier estimate.
    \begin{lemma}\label{triangle_inequality_lemma}
        Let $\Kb$ and $\Kb^*$ satisfy condition \eqref{kernelmass} of Definition \ref{kernelconditions}. Then for all $1\leq p\leq \infty,$ we have
        \[
            \|\Op(\Kb_{\alpha})\|_{L_p(X,\mu)\to L_{p}(X,\mu)} \leq C_2(\Kb)\sum_{j\in \Ntrl} |\alpha_j|.
        \]
    \end{lemma}
    \begin{proof}
        By subadditivity, it is enough to prove the result for a single $K_j.$ The assumption \eqref{kernelmass} asserts the result for $p=\infty,$ and since the same holds for $\Kb^*$ we also have the result for $p=1.$ The general
        case follows from interpolation.
    \end{proof}

\section{Kernels derived from graded bases}\label{AMY_section}
Let $d\geq 1,$ and let $U$ be an open subset of $\Rl^d.$ Let $N\geq d,$ and let $\{X_1,\ldots,X_N\}$ be a family of smooth vector fields on $U$ which span the tangent space at every point of $U.$ Let $w_1\leq \ldots\leq w_N$ be positive integers. Let $|\cdot|$ be the usual (Euclidean) norm on $\Rl^d.$
\subsection{Metric associated to a family of vector fields}
Let $V=\Rl^N = \mathrm{span}\{e_1,\ldots,e_N\}$ be an $N$-dimensional vector space, and equip $V$ with the family of dilations defined by
\[
    \delta_t = \bigoplus_{j=1}^N t^{w_j},\quad t > 0.
\]
That is, $\delta_te_j = t^{w_j}e_j.$

Fix a norm $|\cdot|_V$ on $V$ which is homogeneous with respect to $\delta,$ i.e. $|\delta_tv|_V = t|v|_V$ for all $t>0.$ 
Let $Q = \sum_{j=1}^N w_j.$ Note that $Q\geq N\geq d.$ 
We will write $B_V(0,r)$ for the ball in $V$ with respect to the homogeneous norm $|\cdot|_V.$ 
Since $\mathrm{det}(\delta_t) = t^Q,$ the Lebesgue measure of $B_V(0,r)$ is proportional to $r^Q.$

Let $\sharp$ denote the linear map from $V$ to the Lie algebra of vector fields on $U$ given by
\[
    \sharp(e_j) = X_j,\quad 1\leq j\leq N.
\]

For any $x\in U,$ there exists $\varepsilon_x>0$ such that the map defined by
\[
    \Lambda_x:B_V(0,\varepsilon_x)\to U,\, \Lambda_x(z) := \exp(\sharp z)x.
\]
exists and is a submersion. Here, $\exp$ is the exponential of vector fields (i.e., the flow at time $1$). We can see that $\Lambda_x(z)$ exists for sufficiently small $|z|_V$ because $\sharp z$ is smooth. It is also clear that the derivative of $\Lambda_x(z)$ at $z=0$ is surjective since $\{X_1,\ldots,X_N\}$ spans the tangent space at $x.$ Hence, $\Lambda_x(z)$ exists and is a submersion for sufficiently small $|z|_V.$ For further details, see \cite[Page 39]{Goodman-lnm-1976}. 

Since $\Lambda_x$ is a submersion, the pre-image $\Lambda_x^{-1}(y)$ of a point $y\in U$ is a (possibly empty) $(N-d)$-dimensional submanifold of $B_V(0,\varepsilon_x).$ Note that whether or not $\Lambda_x^{-1}(y)$ is empty will depend on the choice of $\varepsilon_x.$
Let $D\Lambda_x$ be the Jacobian matrix of $\Lambda_x;$ since $\Lambda_x$ is a submersion, $D\Lambda_x(z)$ is a $d\times N$-matrix of full rank. Define
\[
    |D\Lambda_x(z)| := \det(D\Lambda_x(z)D\Lambda_x(z)^{\top})^{\frac12},\quad z \in B_V(0,\varepsilon_x).
\]
Here, $\top$ is the matrix transpose.
The co-area formula asserts that
\[
    \int_V g(z)f(\Lambda_x(z))\,dz = \int_{U} f(y)\int_{\Lambda_x^{-1}(y)} g(z)|D\Lambda_x(z)|^{-1}\,d\Hc^{N-d}(z)\,dy,\quad g \in C^\infty_c(B_V(0,\varepsilon_x)),\, f \in C^\infty_c(U).
\]
Here $\Hc^{N-d},$ is the $(N-d)$-dimensional Hausdorff measure on the submanifold $\Lambda_x^{-1}(y).$
See \cite[Theorem 3.2.12]{Federer-1969}. We will frequently refer to this measure and so we introduce the following notation.
\begin{definition}\label{disintegrated_measure_definition}
    Let $x,y\in U.$ Let $\mu^{x,y}$ denote the following measure on $\Lambda_x^{-1}(y)\subset B_V(0,\varepsilon_x)$
    \[
        d\mu^{x,y}(z) := \frac{d\Hc^{N-d}(z)}{|D\Lambda_x(z)|}.
    \]

\end{definition}
With this notation, the coarea formula for $\Lambda_x$ asserts that
\begin{equation}\label{coarea_formula}
    \int_V g(z)f(\Lambda_x(z))\,dz = \int_U f(y)\int_{\Lambda_x^{-1}(y)} g(z)\,d\mu^{x,y}(z)\,dy,\quad g \in C^\infty_c(B_V(0,\varepsilon_x)),\, f \in C^\infty_c(U).
\end{equation}

We record a few essential features of $\mu^{x,y}.$
\begin{lemma}\label{symmetry_of_Lambda}
    Let $x,y \in U.$ We have
    \[
        \Lambda_x^{-1}(y)\cap B_V(0,\min\{\varepsilon_x,\varepsilon_y\}) = -\Lambda_y^{-1}(x)\cap B_V(0,\min\{\varepsilon_x,\varepsilon_y\})
    \]
    and $d\mu^{x,y}(z) = d\mu^{y,x}(-z)$ on $B_V(0,\min\{\varepsilon_x,\varepsilon_y\}).$
\end{lemma}
\begin{proof}
    By definition, $z \in \Lambda_x^{-1}(y)$ if and only if $|z|_V < \varepsilon_x$ and $\exp(\sharp z)x=y.$ Since $\exp(\sharp (-z))\exp(\sharp z)x = x,$ we have $\exp(\sharp (-z))y = x,$ and hence if $|-z|<\varepsilon_y,$ we have     
    $-z\in \Lambda_y^{-1}(x).$ This proves that if $|z|_V < \min\{\varepsilon_x,\varepsilon_y\}$ and $z \in \Lambda_x^{-1}(y),$ then $-z\in \Lambda_x^{-1}(y).$ 
    By symmetry, we conclude the converse, and this proves the first claim.   

    The fact that $d\mu^{x,y}(z)=d\mu^{y,x}(-z)$ on $B_V(0,\min\{\varepsilon_x,\varepsilon_y\})$ follows from \eqref{coarea_formula}.
\end{proof}

Using $\Lambda$ we can also define the following function on $U.$
\begin{definition}\label{definition_of_the_metric}
    Let $x,y\in U.$ If $y \in \Lambda_x(B_V(0,\varepsilon_x)),$ define
    \[
        \rho(x,y) := \inf\{\lambda>0\;:\; y \in \Lambda_x(B_V(0,\lambda))\}.
    \]
    Otherwise, $\rho(x,y)=\infty.$
\end{definition}
Equivalently, $\rho(x,y)\leq C<\infty$ if $C\leq \varepsilon_x$ and if there exists $z \in \overline{B_V(0,C)}$ such that $\Lambda_x(z)=y.$ Observe that $\rho$ is precisely the function denoted $\rho_2$ in \cite[Section 4]{NagelSteinWainger1985}. The same metric was used by Street, see \cite[Section 1.3]{Street-maximal-subellipticity-2023}.

We want to use $\rho$ as a quasi-metric on $U.$ As defined, $\rho$ may fail to be a quasi-metric for the obvious reason that $\rho(x,y)$ may be infinite. A more subtle problem is that $\rho$ is not necessarily symmetric: it might be that $y\in \Lambda_x(B_V(0,\varepsilon_x))$ but $x\notin \Lambda_y(B_V(0,\varepsilon_y)).$ 

We deal with this issue by assuming that $U$ is as small as necessary. Theorem 2 of \cite{NagelSteinWainger1985} implies the following:
\begin{theorem}\label{everything_can_be_connected}
    Every point $x\in U$ has a neighbourhood $U_0\subset U$ such that $\rho(y_0,y_1)<\infty$ for all $y_0,y_1\in U_0,$ and $\inf_{y\in U_0} \varepsilon_y >0.$ Moreover, $\rho(y_0,y_1)=\rho(y_1,y_0)$ for all $y_0,y_1\in U.$
\end{theorem}

Henceforth, replacing $U$ with $U_0$ if necessary, we will make the following assumption:
\begin{assumption}\label{everything_is_connected}
Assume that $U$ is small enough so that $\rho$ is finite and symmetric on $U,$ and also so that $\inf_{x\in U} \varepsilon_x >0.$

Replacing $\varepsilon_x$ by $\inf_{x\in U}\varepsilon_x,$ we assume without loss of generality that $\varepsilon=\varepsilon_x$ is constant for $x\in U.$ 

We also assume that $U$ has compact closure, in order that all of the derivatives of $\Lambda_x$ with respect to $x$ are uniformly bounded over $x \in U.$

\end{assumption}
Ultimately there will be no harm in making this assumption, since our final estimates will take place over compact subsets of $U$ which can be covered by a finite number of sets satisfying Assumption \ref{everything_is_connected}.

Using this new assumption, $\Lambda_x^{-1}(y)$ is a non-empty submanifold of $B_V(0,\varepsilon)$ for every $x,y\in U.$

Now that $\rho(x,y)$ is guaranteed to be finite and symmetric on $U,$ we can use the following theorem.
\begin{theorem}[Nagel-Stein-Wainger] \cite[Section 4]{NagelSteinWainger1985}
    As given in Definition \ref{definition_of_the_metric}, $\rho$ defines a quasi-metric on $U$
    which is uniformly comparable to the Euclidean metric in the sense that there exist constants $0<c<C<\infty$ such that
    \begin{equation}\label{NSW_distance_upper_bound}
        c|x-y| \leq \rho(x,y)\leq C|x-y|^{\frac{1}{w_N}},\quad x,y\in U.
    \end{equation}
    If $\mu$ denotes the Lebesgue measure on $U,$ then $(U,\rho,\mu)$ is a doubling quasi-metric measure space.
\end{theorem}

We will need some estimates on integrals with respect to the measures $\mu^{x,y}.$
\begin{lemma}\label{disintegration_of_a_ball}
    Let $r>0,$ and
    \[
        F(x,y) := \int_{\Lambda_x^{-1}(y)} \chi_{B_V(0,r)}(z)\,d\mu^{x,y}(z).
    \]
    \begin{enumerate}
    \item{}\label{uniform_integrability_of_integral_of_ball} There exists a constant $C_U>0$ such that
    \[
        \int_U F(x,y)\,dy \leq C_Ur^{Q}
    \]
    and also
    \[
        \int_U F(x,y)\,dx\leq C_Ur^Q.
    \]
    \item{}\label{uniform_support_of_integral_of_ball} If $r<\rho(x,y),$ then $F(x,y)=0.$
    \end{enumerate}
    
\end{lemma}
\begin{proof}
    If $r < \varepsilon,$ then \eqref{symmetry_of_Lambda}, implies that $F(x,y) = F(y,x),$ so
    \[
        \sup_{x\in U} \int_U F(x,y)\,dy = \sup_{y \in U} \int_U F(x,y)\,dx.
    \]
    It follows that for small $r,$ the two claims in \eqref{uniform_integrability_of_integral_of_ball} are equivalent. The same follows for general $r$ by increasing the constant $C_U.$

    By the coarea formula \eqref{coarea_formula}, if $r<\varepsilon_x,$ then
    \[
        \int_U F(x,y)\,dy = \int_{U} \chi_{B_V(0,r)}(z)\,dz \leq C_Vr^Q.
    \]
    By increasing $C_U$ if necessary, this still holds for $r\geq \varepsilon.$
    This proves \eqref{uniform_integrability_of_integral_of_ball}.

    For \eqref{uniform_support_of_integral_of_ball}, note that by definition if $r<\rho(x,y)$ we have $\Lambda_x^{-1}(y)\cap B_V(0,r)=\emptyset,$ and so $F(x,y)=0.$
\end{proof}

\subsection{Graded basis}
Let $U$ be as above. In particular, we continue to enforce Assumption \ref{everything_is_connected}.

Let $\Ubb$ be an open subset of $U\times V\times \Rl$ containing $U\times \{0\}\times \{0\}$ which is invariant under the $\Rl^{\times}$-action $\alpha$ defined by
\[
    \alpha_{\lambda}(x,z,\hbar) = (x,\delta_{\lambda}z,\lambda^{-1}\hbar),\quad (x,z,\hbar)\in \Ubb,\; \lambda>0.
\]
The $4$-tuple $(U,V,\sharp,\Ubb)$ is said to a graded basis for the filtration defined by $\{X_1,\ldots,X_N\},$ in the terminology of \cite[Definition 2.1]{AMY2022arxiv}.

\begin{definition}
Let $r$ denote the projection map 
\[
    r:\Ubb\to U\times \Rl,\quad r(x,z,\hbar) = (x,\hbar).
\]
A function $f$ on $\Ubb$ is said to be $r$-properly supported if $r$ is proper on the support of $k.$ That is, if for any compact subset $A$ of $U\times \Rl,$ $\mathrm{supp}(f)\cap r^{-1}(A)$ is compact.
\end{definition}

\begin{remark}
    If $\Ubb = U\times V\times \Rl,$ then saying that $f$ is $r$-properly supported is simply saying that it is compactly supported in the $V$-variable. For example, $f \in C^\infty(U\times V\times \Rl)$ is $r$-properly supported precisely when we can identify $f$ with a function
    \[
        f\in C^\infty(U\times \Rl,C^\infty_c(V)).
    \]
\end{remark}

\begin{definition}\label{kernel_parts_definition}
    Let $f \in C^\infty(\Ubb)$ be $r$-properly supported. 
    Assume that if $x\in U$ and $|\hbar|<\varepsilon,$ then $z\mapsto f(x,z,\hbar)$ is supported in $|z|_V<1.$

    For fixed $0<\hbar<\varepsilon,$ and $j\in \Ntrl,$ define
    \[
         K_j(x,y) := 2^{jQ}\hbar^{-Q}\int_{\Lambda_x^{-1}(y)} f(x,\delta_{\hbar}^{-1}\delta_{2^j}z,2^{-j}\hbar)\,d\mu^{x,y}(z),\quad x,y\in U.
    \]

    That is, $K_j(x,\cdot)$ is a function on $U$ such that for any $u \in C^\infty_c(U)$ we have
    \[
        \int_{U} K_j(x,y)u(y)\,dy = 2^{jQ}\hbar^{-Q}\int_{V} f(x,\delta_{\hbar}^{-1}\delta_{2^{j}}z,2^{-j}\hbar)u(\Lambda_x(z))\,dz.
    \]
    Let $\Kb = \{K_j\}_{j\in \Ntrl}.$
\end{definition}
Our goal is to show that $\Kb$ in Definition \ref{kernel_parts_definition} obeys all of the conditions of Definition \ref{kernelconditions}, with $(X,\rho,\mu) = (U,\rho,\mu)$ where $\mu$ is the Lebesgue measure. Conditions \eqref{kernelsupport}, \eqref{kernelmass} and \eqref{kernelcancellation} are easy consequences of the Lemma \ref{disintegration_of_a_ball}, while Condition \ref{kernelsmoothnessballvolume} will be more challenging.

\begin{remark}\label{mean_zero_remark}
An additional condition that we may place on $f$, besides Definition \ref{kernel_parts_definition}, is that there exists $C_f>0$ such that 
\[
    \left|\int_V f(x,z,\hbar)\,dz\right| \leq C_f\hbar.
\]
This is equivalent to
\[
    \int_V f(x,z,0)\,dz = 0.
\]
Indeed, since $f$ is smooth there exists an $r$-properly supported smooth function $g$ such that
\[
    f(x,z,\hbar) = f(x,z,0)+\hbar g(x,z,\hbar),\quad (x,z,\hbar)\in \Ubb
\]
and $\int_{V} g(x,z,\hbar)\,dz = O(1)$ as $\hbar\to 0.$
\end{remark}

\begin{lemma}\label{kernel_parts_satisfy_kernelsupport}
    Let $f$ and $K_j$ be as in Definition \ref{kernel_parts_definition}. There exists a constant $C_{U,f}>0$ such that if
    \[
        C_{U,f}\hbar 2^{-j}<\rho(x,y)
    \]
    then $K_{j}(x,y)=0.$ That is, $\Kb = \{K_j\}_{j\in \Ntrl}$ satisfies Condition \eqref{kernelsupport} of Definition \ref{kernelconditions} with constant $C_1(\Kb) = C_{U,f}\hbar.$
\end{lemma}
\begin{proof}
    By assumption, we have
    \[
        |K_j(x,y)| \leq 2^{jQ}\hbar^{-Q}\|f\|_{\infty}\int_{\Lambda_x^{-1}(y)} \chi_{B_V(0,2^{-j}\hbar)}(z)\,d\mu^{x,y}(z).
    \]
    Hence, the result follows from Lemma \ref{disintegration_of_a_ball}.\eqref{uniform_support_of_integral_of_ball}.
\end{proof}

\begin{lemma}\label{kernel_parts_satisfy_kernelmass}
    Let $f$ and $K_j$ be as in Definition \ref{kernel_parts_definition}. There exists a constant $C_{U,f}>0$, independent of $\hbar$ such that
    \[
        \sup_{j\geq 0} \int_{U} |K_j(x,y)|\,dy \leq C_{U,f}
    \]
    That is, $\Kb = \{K_j\}_{j\in \Ntrl}$ satisfies Condition \eqref{kernelmass} of Definition \ref{kernelconditions}, with constant $C_2(\Kb) = C_{U,f}.$
\end{lemma}
\begin{proof}
    Exactly as in Lemma \ref{kernel_parts_satisfy_kernelsupport}, we have
    \[
        |K_j(x,y)| \leq 2^{jQ}\hbar^{-Q}\|f\|_{\infty}\int_{\Lambda_x^{-1}(y)} \chi_{B_V(0,2^{-j}\hbar)}(z)\,d\mu^{x,y}(z).
    \]
    The result follows from Lemma \ref{disintegration_of_a_ball}.\eqref{uniform_integrability_of_integral_of_ball}.
\end{proof}

\begin{lemma}\label{kernelcancellation_verification}
    Let $f$ and $K_j$ be as in Definition \ref{kernel_parts_definition}. Assume in addition that
    \[
        \int_{V} f(x,z,0)\,dz = 0.
    \]
    There exists a constant $C_{U,f}>0,$ independent of $\hbar,$ such that
    \[
        \left|\int_{U} K_j(x,y)\,dy\right| \leq C_{U,f}2^{-j}\hbar.
    \]
    That is, $\Kb = \{K_j\}_{j\in \Ntrl}$ satisfies Condition \eqref{kernelcancellation} of Definition \ref{kernelconditions}, with constant $C_4(\Kb) = C_{U,f}\hbar.$
\end{lemma}
\begin{proof}
    The condition 
    \[
        \int_{V} f(x,z,0)\,dz = 0
    \]
    implies that
    \[
        \int_V f(x,z,\hbar)\,dz = O(\hbar),\quad \hbar\to 0.
    \]
    (See Remark \ref{mean_zero_remark}).

    By the coarea formula \eqref{coarea_formula} and the assumption on $f,$ we have
    \begin{align*}
        \left|\int_U K_j(x,y)\,dy\right| &= \left|\int_{U} \int_V 2^{jQ}\hbar^{-Q}f(x,\delta_{\hbar}^{-1}\delta_{2^j}z,\hbar)\,d\mu^{x,y}(z)\,dy\right|\\
                            &= \left|\int_{V} f(x,z,2^{-j}\hbar)\,dz\right|\\
                            &\leq C_{U,f} 2^{-j}\hbar.
    \end{align*}
\end{proof}

\subsection{Level sets of submersions}
The final and most difficult condition to verify from Definition \ref{kernelconditions} is Condition \eqref{kernelsmoothnessballvolume}. In order to achieve this, we need to understand how the measure $\mu^{x,y}$ varies with $y\in U.$ This can be done using classical vector calculus.

To simplify the notation, in this subsection we consider a general submersion
\[
    \psi:V\to \Rl^d
\]
where recall that $V=\Rl^N,$ and $N\geq d$ and is equipped with a norm $|\cdot|_V.$ The fact that $V$ is equipped with a family of anisotropic dilations will not be important in this subsection. Ultimately we will consider $\psi = \Lambda_x,$ which is only defined on an open subset of $V,$ but for brevity we work here with $\psi$ defined on the entire space. Since our reasoning is local this is no loss of generality.
Given $y \in \Rl^d,$ let $\mu^y$ denote the measure on $\psi^{-1}(y)$ given by
\[
    d\mu^y(z) = \frac{d\Hc^{N-d}(z)}{|D\psi(z)|}.
\]
The coarea formula asserts that
\begin{equation}\label{simplified_coarea}
    \int_{\Rl^d} g(y)\int_{\psi^{-1}(y)} f(z)\,d\mu^y(z)\,dy = \int_{V} g(\psi(z))f(z)\,dz,\quad f \in C^\infty_c(V),\, g \in C^\infty_c(\Rl^d).
\end{equation}

In this subsection we will derive a formula for the difference
\[
    \int_{\psi^{-1}(y_1)} f(z)\,d\mu^{y_1}(z)-\int_{\psi^{-1}(y_0)} f(z)\,d\mu^{y_0}(z),\quad y_0,y_1\in \Rl^d,\; f \in C^\infty_c(V).
\]
The intent will be to analyse the difference $K_j(x,y)-K_j(x,y_0),$ where $K_j$ is as in Definition \ref{kernel_parts_definition}.

The idea is to define a flow in $V$ which transports $\psi^{-1}(y_0)$ to $\psi^{-1}(y_1).$
Let $W$ denote the vector field on $\Rl^N$ given by
\begin{equation}\label{W_definition}
    W(z) = D\psi(z)^{\top}(D\psi(z)D\psi(z)^{\top})^{-1}(y_1-y_0).
\end{equation}
Since $\psi$ is a submersion and $y_0\neq y_1,$ the vector field $W$ is nowhere vanishing. Let $\{\theta_t\}_{t\in \Rl}$ be the flow of this vector field. That is, $\theta_t(z) = \exp(tW)z,$ or more specifically
$t\mapsto \theta_t(z)$ solves the differential equation
\[
    \frac{d}{dt}\theta_t(z) = D\psi(\theta_t(z))^{\top}(D\psi(\theta_t(z))D\psi(\theta_t(z))^{\top})^{-1}(y_1-y_0),\quad \theta_0(z) = z.
\]
In principle, the flow $t\mapsto \theta_t(z)$ may only exist for a time $|t|<\varepsilon_z.$

If it exists, the map $\theta_1$ transports $\psi^{-1}(y_0)$ to $\psi^{-1}(y_1).$ One way to see this is to define
\[
    y_t := (1-t)y_0+ty_1,\quad 0\leq t\leq 1.
\]
The chain rule gives
\begin{align*}
    \frac{d}{dt}(\psi(\theta_t(z))-y_t) &= D\psi(\theta_t(z))\frac{d}{dt}\theta_t(z)-\frac{d}{dt}y_t\\
    &= (D\psi(\theta_t(z))D\psi(\theta_t(z))^{\top}(D\psi(\theta_t(z))D\psi(\theta_t(z))^{\top})^{-1}-1)(y_1-y_0)\\
    &= 0.
\end{align*}
Hence, if $\psi(z)=y_0,$ then $\psi(\theta_t(z)) = y_t$ for all $0\leq t\leq 1.$
It follows that $\theta_t$ defines a smooth map from $\psi^{-1}(y_0)$ to $\psi^{-1}(y_t).$ In particular, at $t=1$ we have a smooth map
\[
    \theta_1:\psi^{-1}(y_0)\to \psi^{-1}(y_1).
\]
Similar reasoning shows that $\theta_{-1}$ is a smooth inverse for $\theta_1,$ so $\theta_1$ is a diffeomorphism.

The next step is to determine how $\mu^{y}$ changes with $\theta_t.$ Again let $y_t := y_0+t(y_1-y_0).$ We have
\begin{equation}\label{radon_nykodym_formula}
    \frac{d(\theta_t^*\mu^{y_t})}{d\mu^{y_0}} = |D\theta_t|.
\end{equation}
Here, $\theta_t^*\mu^{y_t} = \mu^{y_t}\circ \theta_t$ is the pullback of $\mu^{y_t}$ to $\psi^{-1}(y)$ and $|\cdot|$ is the determinant.
To see \eqref{radon_nykodym_formula}, let $g \in C^\infty_c(\Rl^d)$ and $f\in C^\infty_c(\Rl^N).$ By \eqref{simplified_coarea}, we have
\begin{align*}
    \int_{\Rl^d} g(y)\int_{\psi^{-1}(y_0)} f(z)\,d(\theta_t^*\mu^{y_t})(z)\,dy_0 &= \int_{\Rl^d} g(y)\int_{\psi^{-1}(y_t)} f(\theta_t^{-1}z)\,d\mu^{y_t}(z)\,dy\\
                                                                            &= \int_{\Rl^d} g(y_t-t(y_1-y_0))\int_{\psi^{-1}(y_t)}f(\theta_t^{-1}z)\,d\mu^{y_t}(z)\,dy_t\\
                                                                            &= \int_{\Rl^d} g(\psi(z)-t(y_1-y_0))f(\theta_t^{-1}z)\,dz.
\end{align*}
By a change of variables
\begin{align*}
    \int_{\Rl^d} g(y_0)\int_{\psi^{-1}(y_0)} f(z)\,d(\theta_t^*\mu^{y_t})(z)\,dx &= \int_{\Rl^d} g(\psi(\theta_t(z))-t(y_1-y_0))f(z)|D\theta_t(z)|\,dz\\
                                                                            &= \int_{\Rl^d} g(\psi(z))f(z)|D\theta_t(z)|\,dz.
\end{align*}
Applying \eqref{simplified_coarea} again gives
\[
    \int_{\Rl^d} g(y_0)\int_{\psi^{-1}(y_0)} f(z)\,d(\theta_t^*\mu^{y_t})(z)\,dy_0 = \int_{\Rl^d} g(y_0)\int_{\psi^{-1}(y_0)} f(z)|D\theta_t(z)|\,d\mu^{y_0}(z)\,dy.
\]
Since $g$ and $f$ are arbitrary, \eqref{radon_nykodym_formula} follows.

Using \eqref{radon_nykodym_formula}, we have
\[
    \int_{\psi^{-1}(y_1)} f(z)\,d\mu^{y_t}(z) = \int_{\psi^{-1}(y_0)} f(\theta_1(z))|D\theta_1(z)|\,d\mu^{y_0}(z).
\]
Therefore
\begin{align*}
    \int_{\psi^{-1}(y_0)}f(z)\,d\mu^{y_0}(z)-\int_{\psi^{-1}(y_1)}f(z)\,d\mu^{y_1}(z) &= \int_{\psi^{-1}(y_0)} f(z)-f(\theta_1(z))|D\theta_1(z)|\,d\mu^{y_0}(z)\\
                    &= \int_{\psi^{-1}(y_1)} f(z)-f(\theta_1(z))+f(\theta_1(z))(1-|D\theta_1(z)|)\,d\mu^{y_0}(z).
\end{align*}

Recall that $\theta_t = \exp(tW),$ where $W$ is the vector field \eqref{W_definition}. The Liouville-Ostrogradsky formula\footnote{This formula appears to be folklore. A proof is sketched as an exercise on Page 86 of \cite{Perko2001}.} says that Jacobian determinant of $\theta_t$ obeys the differential equation
\[
    \frac{d}{dt}|D\theta_t(z)| = \mathrm{div}(W)(\theta_t(z))|D\theta_t(z)|.
\]
Since $\theta_0(z)=z,$ we have $|D\theta_0(z)|=1$ and hence
\[
    |D\theta_1(z)| = \exp(\int_0^1 \mathrm{div}(W(\theta_t(z)))\,ds).
\]
Thus,
\begin{align}\label{moving_surfaces_difference_formula}
    &\int_{\psi^{-1}(y_0)}f(z)\,d\mu^{y_0}(z)-\int_{\psi^{-1}(y_1)}f(z)\,d\mu^{y_1}(z) \nonumber\\
    &= \int_{\psi^{-1}(y_0)} f(z)-f(\theta_1(z))\,d\mu^{y_0}(z) + \int_{\psi^{-1}(y_0)} f(\theta_1(z))(1-\exp(\int_0^1 \mathrm{div}(W)(\theta_t(z))\,dt))\,d\mu^{y_0}(z).
\end{align}

With this preparation, we can prove the main result of this subsection.
\begin{proposition}\label{disintegration_is_wasserstein_lipschitz}
    Let $\psi:V\to \Rl^d$ be a submersion, let $R>0.$
    There exists $\delta_{\psi,R}>0$ and $C_{\psi,R}>0$ such that for all $0<r<R$, $|y_0-y_1|\leq \delta_{\psi,R}r$
    and $f \in C^\infty_c(B_V(0,r))$ we have
    \begin{align*}
        &\left|\int_{\psi^{-1}(y_0)}f(z)\,d\mu^{y_0}(z)-\int_{\psi^{-1}(y_1)}f(z)\,d\mu^{y_1}(z) \right|\\
        &\leq C_{\psi,R}(\|f\|_{\mathrm{Lip}(V)}+\|f\|_{L_{\infty}(V)})|y_1-y_0|\int_{\psi^{-1}(y_0)} \chi_{B_V(0,2r)}(z)\,d\mu^{y_0}(z).
    \end{align*}
\end{proposition}
\begin{proof}
    Let $W$ be the vector field in \eqref{W_definition}.
    Since $z\mapsto W(z)$ is smooth and depends linearly on $y_0-y_1,$ for every $K\geq 0$ and and $L>0$ there is a constant $C_{\psi,L,K}$ such that
    \[
        \sup_{|\alpha|\leq K} \sup_{|z|_V<L} |\partial_z^{\alpha}W(z)|_V \leq C_{\psi,L,K}|y_0-y_1|.
    \]
    Let $|z|_V<r<R.$ By the Picard-Lindel\"of theorem, if
    \[
      0<C_{\psi,2R,1}(1+R)\delta_{\psi,R} < 1
    \]
    then for $|y_0-y_1|<r\delta_{\psi,R},$ the flow $\{\theta_t(z)\}_{-1\leq t\leq 1}$ exists and moreover
    \[
        \sup_{-1\leq t\leq 1} |\theta_t(z)|_V \leq 2R.
    \]
    This implies that
    \[
        \sup_{0\leq t\leq 1} |W(\theta_t(z))|_V \leq C_{\psi,2R,1}|y_0-y_1|.
    \]
    
    Similarly,
    \begin{equation}\label{deformation_estimate}
        \sup_{0\leq t\leq 1} |\mathrm{div}(W)(\theta_t(z))| \leq C_{\psi,2R,1}|y_0-y_1|.
    \end{equation}

    Since
    \[
        \theta_1(z)-z = \int_0^1 W(\theta_t(z))\,dt
    \]
    we have
    \begin{equation}\label{distance_estimate}
        |\theta_1(z)-z|_V\leq \sup_{0\leq t\leq 1} |W(\theta_t(z))|_V \leq C_{\psi,2R,1}|y_0-y_1|.
    \end{equation}
    Since $|y_0-y_1|<r\delta_{\psi,R},$ this implies
    \[
        |\theta_1(z)-z|_V \leq C_{\psi,2R,1}\delta_{\psi,R}r < \frac{r}{1+R}<r.
    \]
    Therefore,
    \[
        |\theta_1(z)|_V < 2r.
    \]
    Similarly, $|\theta_{-1}(z)|_V < 2r.$
    
    Let $f \in C^\infty_c(B_V(0,r)).$ Note that $f\circ \theta_1$ is supported in $B_V(0,2r),$ since $|\theta_{-1}(z)|_V < 2r.$ 
    Now we apply \eqref{moving_surfaces_difference_formula}, resulting in
    \begin{align*}
        &\left|\int_{\psi^{-1}(y_0)}f(z)\,d\mu^{y_0}(z)-\int_{\psi^{-1}(y_1)}f(z)\,d\mu^{y_1}(z) \right| \leq \int_{\psi^{-1}(y_0)} |f(z)-f(\theta_1(z))|\,d\mu^{y_0}(z)\\
        &\quad + \int_{\psi^{-1}(y_0)} |f(\theta_1(z))|(1-\exp(\int_0^1 \mathrm{div}(W)(\theta_t(z))\,dt))\,d\mu^{y_0}(z)\\
        &\leq \|f\|_{\mathrm{Lip}(V)}\int_{\psi^{-1}(y_0)} \chi_{B_V(0,2r)}(z)|\theta_1(z)-z|_V\,d\mu^{y_0}(z)\\
        &\quad + \|f\|_{L_{\infty}(V)}\int_{\psi^{-1}(y_0)} \chi_{B_V(0,2r)}(z)|1-\exp(\int_0^1 \mathrm{div}(W)(\theta_t(z)))|\,d\mu^{y_0}(z).
    \end{align*}
    By \eqref{distance_estimate}, \eqref{deformation_estimate}
    and the numerical inequality
    \[
        |1-\exp(x)| \leq |x|\exp(|x|),\quad x \in \Rl
    \]
    we have
    \begin{align*}
         &\left|\int_{\psi^{-1}(y_0)}f(z)\,d\mu^{y_0}(z)-\int_{\psi^{-1}(y_1)}f(z)\,d\mu^{y_1}(z) \right|\\
         &\leq C_{\psi,2R,1}\|f\|_{\mathrm{Lip}(V)}|y_0-y_1|\int_{\Lambda_x^{-1}(y_0)} \chi_{B_V(0,2r)}(z)\,d\mu^{y_0}(z)\\
         &\quad +C_{\psi,2R,1}\|f\|_{L_{\infty}(V)}|y_0-y_1|\exp(C_{\psi,2R,1}|y_0-y_1|)\int_{\psi^{-1}(y_0)} \chi_{B(0,2r)}(z)\,d\mu^{y_0}(z).
    \end{align*}
    Since $|y_0-y_1|\leq \delta_{\psi,R} < \frac{1}{C_{\psi,2R,1}},$ we conclude the result with constant $C_{\psi,R} = (1+e)C_{\psi,2R,1}.$
\end{proof}

\subsection{The proof of \eqref{kernelsmoothnessballvolume} for $\Kb$}
\begin{theorem}\label{hard_part}
    Let $\Kb = \{K_j\}_{j\in \Ntrl}$ be the sequence in Definition \ref{kernel_parts_definition}. There exist constants $M_{U}>0$ and $C_{U,f}>0$ such that if $\rho(y,y_0)\leq M_{U}2^{-j},$ then
    \[
        |K_j(x,y)-K_j(x,y_0)| \leq 2^{j}\rho(y,y_0)I_j(x,y),\quad x,y,y_0\in U,\; j\in \Ntrl
    \]
    where
    \[
        I_j(x,y) = C_{U,f}\hbar^{-1}2^{jQ}\hbar^{-Q}\int_{\Lambda_x^{-1}(y)} \chi_{B_V(0,2^{1-j}\hbar)}(z)\,d\mu^{x,y}(z).
    \]
\end{theorem}
\begin{proof}
    For brevity, let
    \[
        f_j(z) := f(x,\delta_{2^j}\delta_{\hbar}^{-1}z,2^{-j}\hbar).
    \]
    Recall that $f_j$ is supported in $B_V(0,2^{-j}\hbar),$ and since $j\geq 0$ and $\hbar<1,$ there is a constant $C_f$ such that
    \[
        \|f_j\|_{\mathrm{Lip}(V)}+\|f_j\|_{L_{\infty}(V)} \leq C_f2^{j}\hbar^{-1}.
    \]    
    By the definition of $K_j,$
    \[
        K_j(x,y) = 2^{jQ}\hbar^{-Q}\int_{\Lambda_x^{-1}(y)} f_j(z)\,d\mu^{x,y}(z),\quad x,y\in U.
    \]
    Now we apply Proposition \ref{disintegration_is_wasserstein_lipschitz}
    with $\psi = \Lambda_x,$ $f=f_j$ and $R=1$ and $r=2^{-j}\hbar.$ 
    If $\delta = \delta_{\Lambda_x,1}$ is chosen as in Proposition \ref{disintegration_is_wasserstein_lipschitz}, then for $|y-y_0|<\delta 2^{-j}\hbar,$ we have
    \begin{equation*}
        |K_j(x,y)-K_j(x,y_0)|\leq C_{f}C_{\Lambda_x,1}2^j|y-y_0|\cdot 2^{jQ}\hbar^{-Q}\int_{\Lambda_x^{-1}(y)}\chi_{B_V(0,2^{1-j}\hbar)}(z)\,d\mu^{x,y}(z).
    \end{equation*}
    Since $\Lambda_x$ depends smoothly on $x,$ and $U$ has been chosen as in Assumption \ref{everything_is_connected}, the constants $C_{\Lambda_x,1}$ and $\delta_{\Lambda_x,1}$ can be chosen uniform over $x \in U.$ 
    By  \eqref{NSW_distance_upper_bound}, we can replace the Euclidean distance $|y-y_0|$ by $\rho(y,y_0),$ 
    Hence, for $|y-y_0|<\delta 2^{-j}\hbar,$ we have
    \[
        |K_j(x,y)-K_j(x,y_0)|\leq 2^j\rho(y,y_0)I_j(x,y).
    \]
    By \eqref{NSW_distance_upper_bound}, this still holds if $\rho(y,y_0)<M_{U}2^{-j}\hbar$ for sufficiently small $M_{U},$ and this completes the proof.

\end{proof}

\begin{corollary}\label{all_the_conditions_hold}
    Let $\Kb = \{K_j\}_{j\in \Ntrl}$ be the sequence in Definition \ref{kernel_parts_definition}.
    There exists constants $C_{U,f},M_{U}>0$ such that $\Kb$ satisfies Condition \eqref{kernelsmoothnessballvolume} of Definition \ref{kernelconditions} with constants
    $M(\Kb) = M_{U}\hbar$ and $C_{3}(\Kb) = C_{U,f}\hbar^{-1}.$
\end{corollary}
\begin{proof}
    By Lemma \ref{disintegration_of_a_ball}, there exists a constant $C_U$ such that function $I_j$ defined in Theorem \ref{hard_part} satisfies
    \[
        \sup_x \int_{U} I_j(x,y)\,dy,\, \sup_y \int_U I_j(x,y)\,dx \leq C_UC_{U,f}\hbar^{-1}.
    \]
    This verifies \eqref{kernelsmoothnessballvolume}.
\end{proof}

The results of this section are summarised in the following theorem:
\begin{theorem}
    Let $f$ and $\Kb = \{K_j\}_{j\in \Ntrl}$ be as in Definition \ref{kernel_parts_definition}. The sequence $\Kb$ obeys Conditions \eqref{kernelsupport}, \eqref{kernelmass}, \eqref{kernelsmoothnessballvolume} of Definition \ref{kernelconditions}, with constants
    \[
        C_1(\Kb) = C_{U,f}\hbar,\quad C_2(\Kb) = C_{U,f},\quad C_3(\Kb) = C_{U,f}\hbar^{-1}
    \]
    and $M(\Kb) =C_{U,f}\hbar$ as $\hbar\to 0.$
    
    If
    \[
        \int_V f(x,z,0)\,dz = 0
    \]
    then $\Kb$ satisfies Condition \eqref{kernelcancellation}
    with constant $C_4(\Kb) = C_{U,f}\hbar$ as $\hbar\to 0.$
\end{theorem}
The important point is that although $C_3(\Kb)\to\infty$ and $M(\Kb)\to 0$ as $\hbar\to 0,$ the product $C_1(\Kb)C_3(\Kb)$ and the ratio $C_1(\Kb)/M(\Kb)$ are bounded, and hence the upper bounds in the theorems of Section \ref{abstract_operator_section} are uniformly bounded as $\hbar\to 0.$ This will be crucial in the proof of the uniform continuity statement of Theorem \ref{main_theorem}.

\section{The Androulidakis-Mohsen-Yuncken calculus}\label{AMY_operator_section}
In this section we recall the definition of the Androulidakis-Mohsen-Yuncken calculus, and prove Theorem \ref{main_theorem}.
Let $(U,V,\sharp,\Ubb)$ be as in the previous section. That is, $U\subset \Rl^d$ is open, $V$ is an $N$-dimensional vector space equipped with dilations $\{\delta_t\}_{t>0},$ $\Ubb$ is an open subset of $U\times V\times \Rl$ containing $U\times \{0\}\times \{0\}$ and invariant under the action
\[
    \alpha_{\lambda}(x,z,\hbar) = (x,\delta_{\lambda}z,\lambda^{-1}\hbar),\quad \lambda>0,\; (x,z,\hbar)\in \Ubb
\]
and $\sharp$ is a linear map from $V$ to vector fields on $U.$ 

As previously, $r:\Ubb\to U\times \Rl$ is defined as $r(x,z,\hbar) = (x,\hbar).$ Note that $r$ is surjective.

We consider distributions on $\Ubb$ which are $r$-fibred and $r$-proper. The set of such distributions is written $\Ec'_r(\Ubb).$ Equivalently, $k\in \Ec'_r(\Ubb)$ is a linear map
\[
    k:C^\infty(\Ubb)\to C^\infty(U\times \Rl)
\]
which is $C^\infty(U\times \Rl)$-linear, where $C^\infty(\Ubb)$ is a $C^\infty(U\times \Rl)$-module in the obvious way.

We will consider distributions as generalised functions rather than generalised densities. That is, if $k$ is a smooth function on $\Ubb,$ then $k$ determines an $r$-fibred distribution by
\[
    (k,\phi)(x,\hbar) = \int_{V} k(x,z,\hbar)\phi(x,z,\hbar)dz,\quad \phi \in C^\infty(\Ubb),\; (x,\hbar) \in U\times \Rl
\]
where $dz$ is the Lebesgue measure on $V.$ This differs from the convention in \cite{vanErpYuncken2019}, where the measure $dz$ was incorporated into the distribution.

Given an $r$-fibred distribution $k\in \Ec'_r(\Ubb),$ the pushforward $(\alpha_{\lambda})_*k$ is defined on functions $k$ by $k\circ \alpha_{\lambda}^{-1},$ and on $r$-fibred distributions $k\in \Ec'_r(\Ubb)$ by
\[
    ((\alpha_{\lambda})_*k,\phi)(x,\hbar) = \lambda^{Q}(k,\phi\circ \alpha_{\lambda})(x,\lambda\hbar),\quad \phi\in C^\infty(\Ubb)\; (x,\hbar) \in r(\Ubb).
\]
\begin{remark}
    If $\Ubb = U\times V\times \Rl,$ then $\Ec'_{r}(\Ubb)$ can be described slightly more simply. We have a continuous identification
    \[
        \Ec'_{r}(U\times V\times \Rl) = C^\infty(U\times \Rl,\Ec'(V)).
    \]
    That is, $k\in \Ec'_r(U\times V\times \Rl)$ if $(x,z,\hbar)\to k(x,z,\hbar)$ is smooth in $(x,\hbar) \in U\times \Rl,$ and compactly supported in $z \in V.$
\end{remark}

An $r$-fibred distribution $k\in \Ec'_r(\Ubb)$ defines a family of linear operators $\{\Pb_\hbar\}_{\hbar>0}$ from $C^\infty_c(U)$ to $C^\infty(U)$ by
\[
    \Pb_{\hbar}u(x) = \int_{V} k(x,z,\hbar)u(\exp(\sharp \delta_{\hbar}z)x)\,dz = \hbar^{-Q}\int_V k(x,\delta_{\hbar}^{-1}z,\hbar)u(\exp(\sharp z)x)\,dz.
\]
This was denoted $(\mathrm{ev}_{\hbar})_*k$ in \cite{AMY2022arxiv}. Note that \emph{a priori} this formula may not make sense for all $x\in U$ and $u \in C^\infty_c(U),$ because the exponential $\exp(\sharp \delta_{\hbar}z)x$ may not be defined. For this reason we assume that $0<\hbar<1$ and restrict attention to a set $U$ satisfying Assumption \ref{everything_is_connected}. 

By the co-area formula, we have
\begin{equation}\label{operator_associated_to_kernel}
    \Pb_{\hbar}u(x) = \int_U \int_{\Lambda_x^{-1}(y)} k(x,z,\hbar)|D\Lambda_x(z)|^{-1}\,d\Hc^{N-n}(z) u(y)\,dy.
\end{equation}
That is, with respect to the Lebesgue measure on $U,$ $\Pb_{\hbar}$ has integral kernel
\[
    K(x,y) = \hbar^{-Q}\int_{\Lambda_x^{-1}(y)} k(x,\delta_{\hbar}^{-1}z,\hbar)\,d\mu^{x,y}(z),\quad x\neq y
\]
where $\mu^{x,y}$ is the same as in Definition \ref{disintegrated_measure_definition}.
Note that if $z\mapsto k(x,z,\hbar)$ is smooth away from $z=0,$ formula makes sense for every $x\neq y.$

The following is an adaptation of \cite[Definition 3.3]{AMY2022arxiv}.
\begin{definition}
    Let $(U,V,\sharp,\Ubb)$ be as above. For $m\in \Cplx,$ the set $\Ec'^{m}(\Ubb)$ is the subspace of $r$-fibred distributions $k\in \Ec'_r(\Ubb)$ such that
    for all $\lambda>0$ we have
    \[
        \lambda^{-m-Q}(\alpha_{\lambda})_*k-k \in C^\infty(\Ubb).
    \]
\end{definition}
This is an ``adaptation" of the definition in \cite{AMY2022arxiv} rather than a direct translation because the operators there were considered to act on half-densities, whereas here we consider operators acting on functions. This causes a shift by $Q$ in the exponent of $\lambda,$ but this does not change the class of operators.

\begin{definition}\label{AMY_operators_definition}
    For $m\in \Cplx,$ the set $\Psi^m_{\Fc}(U)$ of pseudodifferential operators associated to the filtration defined by $\{X_1,\ldots,X_N\}$
    with weights $\{w_1,\ldots,w_N\}$ is the set of operators of the form $\Pb_1,$ where
    \[
        \Pb_{\hbar}u(x) = \int_U \int_{\Lambda_x^{-1}(y)} k(x,z,\hbar)|D\Lambda_x(z)|^{-1}\,d\Hc^{N-n}(z) u(y)\,dy,\quad, u\in C^\infty_c(U), \hbar>0
    \]
    and $k\in \Ec'^{m}(\Ubb).$
    
    Given $T \in \Psi^m_{\Fc}(U),$ an associated family $\{T_{\hbar}\}_{\hbar>0}$ is any choice $\{\Pb_{\hbar}\}_{\hbar>0}$ defined by \eqref{operator_associated_to_kernel} for which $k\in \Ec'^{m}(\Ubb)$ and $T=\Pb_1.$    
\end{definition}

For $1\leq p < \infty,$ $L_{p}(U)$ denotes the $L_p$-space with respect to the Lebesgue measure on $U,$ that is
\[
    \|u\|_{L_p(U)} :=\left(\int_U |u(x)|^p\,dx\right)^{\frac1p}.
\]
As usual this should be replaced with the space of essentially bounded functions on $U$ when $p=\infty.$

\begin{definition}
    For $1\leq p\leq \infty$, we say that an operator $T$ is \emph{locally bounded on }$L_p(U)$ if for all smooth compactly supported functions $\phi,\psi\in C^\infty_c(U),$ there is a constant $C_{T,\phi,\psi}$ such that
    \[
        \|\phi T(\psi u)\|_{L_p(U)} \leq C_{T,\phi,\psi}\|u\|_{L_p(U)},\quad u \in C^\infty_c(U).
    \]
    Similarly, say that $T$ is locally compact if $M_{\phi}TM_{\psi}$ is compact on $L_p(U)$ for all $\phi,\psi\in C^\infty_c(U).$
    We say that a family of operators $\Tb = \{T_{\hbar}\}_{\hbar>0}$ is uniformly continuous as $\hbar\to 0$
    if the constant $C_{T_{\hbar},\phi,\psi}$ can be chosen bounded above as $\hbar\to 0.$ That is, if there exists a constant $C_{\Tb,\phi,\psi}$ such that
    \[
        \limsup_{\hbar\to 0} \|\phi T_{\hbar}(\psi u)\|_{L_p(U)} \leq C_{\Tb,\phi,\psi}\|u\|_{L_p(U)},\quad u \in C^\infty_c(U).
    \]
\end{definition}

The following theorem is a restatement of Theorem \ref{main_theorem}, with the notation described above.
\begin{theorem}\label{AMY_zero_order_bounded}
    Let $k \in \Ec'^{m}(\Ubb)$ where $\Re(m)=0,$ and let $\Pb_{\hbar}$ be the corresponding operator for $\hbar<1,$ defined as in \eqref{operator_associated_to_kernel}. Then $\Pb_{\hbar}$ is locally bounded on $L_p(U)$ for all $1<p<\infty,$ and the family $\{\Pb_{\hbar}\}_{0<\hbar<1}$ is uniformly continuous as $\hbar\to 0.$
\end{theorem}
\begin{proof}
    Since we only seek local boundedness, we can assume that $U$ is small enough so that Assumption \ref{everything_is_connected} holds. 
        
    Let $f(x,z,\hbar) = 2^{-Q-m}k(x,\delta_2^{-1}z,2\hbar)-k(x,z,\hbar).$ By the definition of $\Ec'^{m}(\Ubb),$ $f \in C^\infty(\Ubb).$ Without loss of generality, we assume that $f(x,z,\hbar)$ is supported in the set $|z|_V < 1$ for $-1<\hbar<1.$

    For every $n\geq 1,$ we have
    \[
        k(x,z,\hbar)-2^{(m+Q)n}k(x,\delta_{2^n}z,2^{-n}\hbar) = \sum_{j=1}^{n} 2^{jm}\cdot 2^{jQ}f(x,\delta_{2^{j}}z,2^{-j}\hbar).
    \]
    Since $k$ is compactly supported in the $z$-variable, it follows that there is a distribution $h(x,z),$ smooth in $x$ and supported at $z=0$ such that
    \[
        k(x,z,\hbar) = h(x,z)+ \sum_{j=1}^\infty 2^{jm}\cdot 2^{jQ}f(x,\delta_{2^j}z,2^{-j}\hbar).
    \]
    Since $h$ has homogeneity $-Q$ with respect to dilations in the $z$-variable, we have $h(x,z) = h(x)\delta_{0}(z)$ for some $h\in C^\infty(U).$ Therefore, the integral kernel $K$ of $\Pb_{\hbar}$ satisfies
    \[
        K(x,y) = \sum_{j=1}^\infty \alpha_jK_j(x,y),\quad x\neq y
    \]
    where $K_j$ are defined as in Definition \ref{kernel_parts_definition}, $\alpha_j =2^{jm}$ and the sum converges pointwise. By Corollary \ref{all_the_conditions_hold}, the sequence $\Kb = \{K_j\}_{j\geq 1}$ satisfies Conditions \eqref{kernelsupport}, \eqref{kernelmass} and \eqref{kernelsmoothnessballvolume} of Definition \ref{kernelconditions}, with constants
    \[
        C_1(\Kb) = C_{U,f}\hbar,\, C_2(\Kb) = C_{U,f},\, C_3(\hbar) = C_{U,f}\hbar^{-1},\, \, M(\Kb) = C_{U,f}\hbar.
    \]
    
    To verify the same for $\Kb^*,$ we need to use the fact that the pseudodifferential calculus is closed under adjoints. This is proved in \cite[Proposition 3.13]{AMY2022arxiv}, and moreover the adjoint $\Pb_{\hbar}^*$ is related to the $r$-fibred distribution
    \[
        \widetilde{k}(x,z,\hbar) := \overline{k(\Lambda_x(\delta_{\hbar}z),-z,\hbar)}.
    \]
    The corresponding cocycle is given by
    \[
        \widetilde{f}(x,z,\hbar) = 2^{-Q}\widetilde{k}(x,\delta_2^{-1}z,2\hbar)-\widetilde{k}(x,z,\hbar).
    \]
    Since $\widetilde{k} \in \Ec'^{-m}(\Ubb),$ the function $\widetilde{f}$ has all the same properties just proved for $f.$ We have
    \[
        \overline{K_j(y,x)} = \int_{V} 2^{jQ}\hbar^{-Q}\widetilde{f}(x,-\delta_{2^j}\delta_{\hbar}^{-1}z,2^{-j}\hbar)\,d\mu^{x,y}(z).
    \]
    Here we are using Lemma \ref{symmetry_of_Lambda} to replace $d\mu^{y,x}(z)$ by $d\mu^{x,y}(-z).$ Therefore, $\Kb^*$ also satisfies Conditions \eqref{kernelsupport}, \eqref{kernelmass} and \eqref{kernelsmoothnessballvolume} of Definition \ref{kernelconditions}, with the same constants as $\Kb.$

    Assume now that $m=0.$ In this case, we have
    \[
        \int_{V} f(x,z,0)\,dz = \int_{V} 2^{-Q}k(x,\delta_2^{-1}z,0)-k(x,z,0)\,dz = 0.
    \]
    So for $m=0,$ Corollary \ref{all_the_conditions_hold} implies that $\Kb$ satisfies all of the conditions of Definition \ref{kernelconditions}. Hence in this case, Corollary \ref{abstract_L_p_boundedness}
    and an easy limit argument
    proves that
    \[
        \limsup_{\hbar\to 0} \| \Pb_{\hbar}\|_{L_2(U)\to L_2(U)} < \infty.
    \]    
    In general, for $\Re(m)=0$ we can reduce to the case $m=0$ by replacing $\Pb_{\hbar}$ by $\Pb_{\hbar}^*\Pb_{\hbar},$ using the fact that $\Psi^{-m}_{\Fc}\cdot \Psi^{m}_{\Fc}\subset \Psi^0_{\Fc}$ \cite[Proposition 3.14]{AMY2022arxiv}. Thus in the general case we still have
    \[
        \limsup_{\hbar\to 0} \| \Pb_{\hbar}\|_{L_2(U)\to L_2(U)} < \infty.
    \]
    Since $C_1(\Kb)C_3(\Kb)$ and $C_1(\Kb)M(\Kb)^{-1},$ $C_2(\Kb)$ are uniformly bounded as $\hbar\to 0,$ it follows from Corollary \ref{weak_type_corollary} that $\Pb_{\hbar}$ has weak-type $(1,1),$ uniformly as $\hbar\to 0.$ 
    
    Hence, for general $m$ with $\Re(m)=0$ and $1<p<\infty$ we have
    \[
        \limsup_{\hbar\to 0} \| \Pb_{\hbar}\|_{L_p(U)\to L_p(U)} < \infty.
    \]
\end{proof}

Finally, we can conclude the $\Re(m)=0$ part of Theorem \ref{main_theorem}
\begin{theorem}\label{main_theorem_bddness_part}
    Let $1<p<\infty,$ and let $T \in \Psi^m_{\Fc}(X).$ If $\Re(m)\leq 0,$ then $T$ is locally bounded on $L_p(X).$

    Moreover, the associated family $\{T_{\hbar}\}_{\hbar>0}$ is uniformly continuous as $\hbar\to 0$ in the sense that if $\phi,\psi\in C^\infty_c(X),$ then
    \[
        \limsup_{\hbar\to 0} \|M_{\phi}T_{\hbar}M_{\psi}\|_{L_p(X)\to L_p(X)} < \infty.
    \]
\end{theorem}
\begin{proof}
    Theorem \ref{AMY_zero_order_bounded} is exactly this result stated in a coordinate chart $U.$ Since the result is local, we immediately obtain the theorem.
\end{proof}

\section{Compactness of negative order operators}
The boundedness and compactness of negative order operators is much easier than zero order operators, since in this case we can represent the operator as a norm convergent series.
\begin{theorem}\label{AMY_negative_order_bounded}
    Let $1\leq p\leq \infty,$ let $k\in \Ec'^{m}(\Ubb)$ for $\Re(m)<0,$ and let $\Pb_{\hbar}$ be the corresponding operator as in \eqref{operator_associated_to_kernel}. Then $\Pb_{\hbar}$ is compact on $L_p(U)$ and the family $\{\Pb_{\hbar}\}_{\hbar>0}$ is uniformly continuous as $\hbar\to 0.$
\end{theorem}
\begin{proof}
    Since the result is local, we may assume that $U$ satisfies Assumption \ref{everything_is_connected}. Since $p=1$ and $p=\infty$ are not included in Theorem \ref{AMY_zero_order_bounded}, we first show the uniform boundedness of $\Pb_{\hbar}$ on $L_p(U)$ as $h\to 0.$
    Let
    \[
        f(x,z,\hbar) = 2^{-m-Q}k(x,\delta_{2}^{-1}z,2\hbar)-k(x,z,\hbar),\quad (x,z,\hbar)\in \Ubb.
    \]
    By definition $f \in C^\infty(\Ubb),$ and just as in Thoerem \ref{AMY_zero_order_bounded} for every $n\geq 1$ we have
    \[
        k(x,z,\hbar)-2^{(m+Q)n}k(x,\delta_{2^n}z,2^{-n}\hbar) = \sum_{j=1}^{n} 2^{j(m+Q)}f(x,\delta_{2^{j}}z,2^{-j}\hbar).
    \]   
    Since $k$ is compactly supported in the $z$-variable and has negative order, it follows that
    \[
        k(x,z,\hbar) = \sum_{j=1}^\infty 2^{jm} \cdot 2^{jQ}f(x,\delta_{2^j}z,2^{-j}\hbar).
    \]
    Therefore
    \[
        \Pb_{\hbar} = \sum_{j=1}^\infty 2^{jm}\Op(K_j)
    \]
    where $K_j$ is related to $f$ precisely as in Definition \ref{kernel_parts_definition}. By Lemma \ref{kernel_parts_satisfy_kernelmass}, the sequence $\Kb = \{K_j\}_{j=1}^\infty$ satisfies Condition \ref{kernelmass}. The adjoint $\Kb^*$ is an operator of the same form, and hence $\Kb$ satisfies the requirements of Lemma \ref{triangle_inequality_lemma}.
    Hence, by Lemma \ref{triangle_inequality_lemma}, we get
    \[
        \|\Pb_{\hbar}\|_{L_p(U)\to L_p(U)} \leq C_{U,f}\sum_{j=1}^\infty 2^{j\Re(m)}= C_{U,f} \sum_{j=1}^\infty 2^{j\Re(m)}.
    \]
    Since $\Re(m)<0,$ we conclude the uniform boundedness of $\Pb_{\hbar}$ on $L_p(U)$ for $1\leq p\leq \infty$ as $\hbar\to 0.$

    Observe that the integral kernel of $\Op(K_j)$ is smooth and compactly supported. In particular, $\Op(K_j)$ is compact on $L_p(U).$ Since $\Pb_{\hbar}$ is a norm convergent sum of compact operators, it follows that $\Pb_{\hbar}$ itself is compact.
\end{proof}

We conclude with the $\Re(m)<0$ part of Theorem \ref{main_theorem}.
\begin{theorem}
    Let $1\leq p\leq \infty,$ and let $T\in \Psi^m_{\Fc}(X)$ where $\Re(m)<0.$ Then $T$ is locally bounded on $L_p(X)$ and the associated family $\{T_{\hbar}\}_{\hbar>0}$ is uniformly continuous as $\hbar\to 0$ on $L_p(X).$ Moreover $T$ is locally compact.
\end{theorem}
\begin{proof}
    Since the assertion is local, it suffices to work in a single coordinate chart, and in this case the result is Theorem \ref{AMY_negative_order_bounded}.
\end{proof}

\section{Sobolev spaces}\label{sobolev_section}
In this section we obtain results about the (local) boundedness of negative order operators between $L_p$ spaces, essentially we derive Sobolev embedding theorems for the AMY calculus.

\begin{definition}\label{sobolev_spaces}
    Let $(X,\Fc)$ be a filtered manifold, $s \in \mathbb{R}$ and $1\leq p\leq\infty.$ The local Sobolev space $W^{s}_{p,\loc}(X;\Fc)$ is the set of all distributions $u\in \Dc'(X)$ such that $Pu\in L_{p,\loc}(X)$ for any $P \in \Psi^s_{\Fc}(X).$
\end{definition}

\begin{example}\label{hormander_example}
    Let $X_1,\ldots,X_n$ be vector fields on $U\subseteq \Rl^d$ which satisfy the rank $r$ H\"ormander property; that is, the tangent space of $U$ is everywhere spanned by the iterated brackets
    \[
    [X_{i_1},[X_{i_2},\cdots,[X_{i_{k}},X_{i_{k+1}}]\cdots]
    \]
    of depth at most $r.$ Let $\Fc$ be the filtration determined by $\Fc_1 = \mathrm{span}_{C^\infty(U)}(\{X_1,\ldots,X_n\}),$ $\Fc_2 = \Fc_1+[\Fc_1,\Fc_1],$ etc.

    For $1<p<\infty,$ have that $u\in W^{1}_{p,\loc}(U)$ if and only if $X_1u,\ldots,X_nu\in L_{p,\loc}(U).$ To see this, let
    \[
        P = 1+\sum_{j=1}^n X_j^*X_j.
    \]
    By \cite[Theorem 3.30]{AMY2022arxiv}, $P$ has a parametrix $S \in \Psi^{-2}_{\Fc}(U),$ i.e. $SP-1$ is infinitely smoothing. If $X_1u,\ldots,X_nu\in L_{p,\loc}(U),$ then
    \[
        u = Su+\sum_{j=1}^n SX_j^*(X_ju) + (1-SP)u.
    \]
    For any $R \in \Psi^1_{\Fc}(U),$ we have $RSX_j\in \Psi^0_{\Fc}(U),$ and hence by Theorem \ref{main_theorem}, $Ru\in L_{p,\loc}(U).$ So by definition we conclude that $u\in W^{1}_{p,\loc}(U;\Fc).$
\end{example}

As in the previous sections, $U\subset \Rl^d$ is an open set and $X_1,\ldots,X_N$ are vector fields on $U$ equipped with weights $w_1,\ldots,w_N.$ We have defined
\[
    Q := \sum_{n=1}^N w_n.
\]
We assume as before that $U$ satisfies Assumption \ref{everything_is_connected}.
\begin{lemma}\label{weak_kernel_bound}
    Let $k\in \Ec'^{m}(\Ubb).$ Let $j\geq 0,$ and let $K_j$ be related to $k$ as in Definition \ref{kernel_parts_definition}. There exists a constant $C$ such that
    \[
        |K_j(x,y)|\leq C\hbar^{-Q}2^{jQ},\quad x,y\in U
    \]
    uniformly in $0<\hbar<1.$
\end{lemma}
\begin{proof}
    The formula for $K_j$ is
    \[
        K_j(x,y) = 2^{jQ}\hbar^{-Q}\int_{\Lambda_x^{-1}(y)} f(x,\delta_{\hbar}^{-1}\delta_{2^j}z,2^{-j}\hbar)\,d\mu^{x,y}(z),\quad x,y\in U.
    \]
    The family of functions $\{ (x,\delta_{\hbar}^{-1}\delta_{2^j}z,2^{-j}\hbar)\}_{j\geq 0,0<\hbar<1}$ has uniformly bounded $L_{\infty}(\Ubb)$-norm, hence over any compact set of $(x,y)$ there is $C$ such that
    \[
        |K_j(x,y)|\leq C\hbar^{-Q}2^{jQ}.
    \]
\end{proof}

\begin{corollary}\label{interpolation_bound}
    For $1\leq p\leq q\leq  \infty,$ we have
    \[
        \|\Op(K_j)\|_{L_p(U)\to L_{q}(U)} \leq C\hbar^{-Q(\frac1p-\frac1q)}2^{jQ(\frac1p-\frac1q)}.
    \]
\end{corollary}
\begin{proof}
    The norm of $\Op(K_j)$ from $L_1$ to $L_{\infty}$ is simply $\sup_{x,y} |K_j(x,y)|,$ so for $p=1$ and $q=\infty$ we obtain the result from Lemma \ref{weak_kernel_bound}. By Lemma \ref{kernel_parts_satisfy_kernelmass}, we also have
    \[
        \sup_{x,j} \int_{U} |K_j(x,y)|d\mu(y) < \infty,\quad \sup_{x,j} \int_{U}|K_j(x,y)|\,d\mu(x) < \infty
    \]
    and hence $\Op(K_j)$ is bounded from $L_{\infty}(U)$ to $L_{\infty}(U)$ and $L_1(U)$ to $L_1(U),$ uniformly in $j.$ The result follows from interpolation.
\end{proof}

If $\Fc$ is any filtration of vector fields on a manifold $X,$ then near any point we can choose a weighted family of vector fields defining the filtration and hence a corresponding value of $Q.$
This integer $Q$ is not really natural since a different choice of weighted vector fields could define the same filtration with a different value of $Q.$ However, for a given choice of $Q$ we can derive a Sobolev embedding theorem for the AMY calculus.

For this section we will assume that $(X,\Fc)$ is a filtered manifold such that there exists a maximal $Q$ so that near any point of $x$ there exists a weighted family of vector fields generating $\Fc$ with sum of weights at most $Q.$

\begin{lemma}\label{supercritical_Sobolev_embedding_lemma}
    Let $T \in \Psi^m_{\Fc}(X)$ and $1\leq p\leq \infty$ where $\Re(m)<-\frac{Q}{p}.$ Then $T$ maps locally compactly $L_p(X)$ into $C(X).$
\end{lemma}
\begin{proof}
    Since the assertion is local, we may again work in a chart $U.$ Write $T=\Op(K),$ where $K=\sum_{j=0}^\infty 2^{jm}K_j$ as in Theorem \ref{AMY_negative_order_bounded}. By Corollary \ref{interpolation_bound}, we have
    \[
        \|Tu\|_{\infty} \leq \sum_{j=0}^\infty 2^{\Re(m)j}\|\Op(K_j)u\|_{\infty} \lesssim \|u\|_p
    \]
    provided that $\Re(m)<-\frac{Q}{p}.$ Since each summand $K_j$ is smooth and compactly supported, $\Op(K_j)$ is locally compact from $L_p(U)$ to $C(U)$ and the above computation shows that $T,$ truncated to any compact subset of $U,$ is a convergent sum of compact operators.
\end{proof}

The following lemma represents a Hardy-Littlewood-Sobolev inequality for the filtration $\Fc.$ The scheme of the proof is essentially the same as classic interpolation-based proofs of Hardy-Littlewood-Sobolev, as in e.g. \cite[Theorem 1, Page 119]{Stein-1970}.
\begin{lemma}\label{subcritical_sobolev_embedding_lemma}
    Let $T \in \Psi^m_{\Fc}(X),$ where $\Re(m)<0.$ For all $1< p < -\frac{Q}{\Re(m)},$ $T$ is locally bounded from $L_p(X)$ into $L_q(X),$ where
    \[
        \frac{1}{q}=\frac{1}{p}+\frac{\Re(m)}{Q}.
    \]
\end{lemma}
\begin{proof}
    It suffices to work in a single coordinate chart $U.$ We show that $T$ is bounded from $L_p(U)$ to $L_{q,\infty}(U);$ the result will then follow by interpolation. Since we work locally, take
    \[
        K = \sum_{j\geq 0} 2^{mj}K_j.
    \]
    Let $t>0.$ We estimate
    \[
        \mu(\{x \in U\;:\; |\Op(K)u(x)| > t\}.
    \]
    Note that if $t\leq 2C\|u\|_p,$ then Markov's inequality delivers
    \[
        \mu(\{x\in U\;:\;|\Op(K)u(x)|>t\}) \lesssim Ct^{-p}\|u\|_p^p \lesssim t^{-q}\|u\|_p^q.
    \]
    So we concentrate on the case that $t>2C\|u\|_p.$ Choose $N$ such that
    \[
        \sum_{j=0}^{N-1} C2^{j(\Re(m)+Q/p)}\|u\|_p < \frac{t}{2}.
    \]
    Let
    \[
        K=K_++K_-,\quad K_- := \sum_{j=0}^{N-1} 2^{mj}K_j,\quad K_+ = \sum_{j=N}^\infty 2^{mj}K_j.
    \]
    From the choice of $N$ and Corollary \ref{interpolation_bound},
    \[
        \|K_-u\|_{\infty} \leq \frac{t}{2}.
    \]
    Therefore
    \[
        \mu(\{x \in U\;:\; |\Op(K)u(x)| > t\} \leq  \mu(\{x \in U\;:\; |\Op(K_+)u(x)| > t/2\}.
    \]
    By Markov's inequality, it follows that
    \begin{align*}
        \mu(\{x \in U\;:\; |\Op(K_+)u(x)| > t/2\} &\lesssim t^{-p}\left(\sum_{j=N}^{\infty} 2^{\Re(m)j}\|\Op(K_j)u\|_p\right)^p\\
                                                  &\lesssim t^{-p}2^{N\Re(m)p}\|u\|_p^p.
    \end{align*}
    By the choice of $N,$ we have
    \[
        t^{-p}2^{N\Re(m)}\|u\|_p^p \lesssim t^{-q}\|u\|_p^q.
    \]
    Hence
    \[
        \mu(\{x\in U\;:\; |\Op(K)u(x)|>t\}) \lesssim t^{-q}\|u\|_p^q,\quad t>0.
    \]
    That is, $\Op(K)$ maps $L_p$ to $L_{q,\infty},$ and this completes the proof.
\end{proof}

By the definition of the Sobolev spaces $W^{s}_{p,\loc}(X;\Fc),$ Lemmas \ref{supercritical_Sobolev_embedding_lemma} and \ref{subcritical_sobolev_embedding_lemma} can be summarised as follows:
\begin{theorem}\label{sobolev_embedding_theorem}
    Let $(X,\Fc)$ be a filtered manifold with the property that there exists an integer $Q$ so that near any point of $X$ there exists a weighted family of vector fields generating $\Fc$ with sum of weights at most $Q.$

    For all $1\leq p\leq \infty,$ we have a locally compact embedding
    \[
        W^{s}_{p,\loc}(X;\Fc)\subset C(X),\quad s>\frac{Q}{p}
    \]
    while for $1<p<\infty$ and
    \[
        \frac{1}{q}=\frac{1}{p}-\frac{s}{Q}
    \]
    we have a continuous embedding
    \[
        W^{s}_{p,\loc}(X;\Fc)\subset L_{q,\loc}(X).
    \]
\end{theorem}
\begin{proof}
    It was proved in \cite[Proposition 3.32]{AMY2022arxiv} that for every $s$ there exists a pair of operators $P_s \in \Psi^s_{\Fc}(X)$ and $P_{-s}\in \Psi^{-s}_{\Fc}(X)$ such that $P_{-s}P_{s} -1 \in \Psi^{-1}_{\Fc}(X).$ By definition if $u\in W^{s}_{p,\loc}(X),$ then $P_su\in L_{p,\loc}(X)$ and therefore
    \[
        u \in P_{-s} L_{p,\loc}(X)+(1-P_{-s}P_s)L_{p,\loc}(X).
    \]
    That is, $u$ is in the image of an operator of order $-s.$
\end{proof}

\begin{example}
    Let $X_1,\ldots,X_n$ be vector fields on an open set $U\subset \Rl^d$ satisfying the H\"ormander condition of rank $r\geq 1,$ as in Example \ref{hormander_example}.

    Let $Q$ be the graded dimension of the free nilpotent Lie algebra $\mathcal{G}_{n,r}$ of depth $r$ with $n$ generators.
    Then for $p<Q$ and $u\in C^\infty_c(U)$ supported on a compact subset $K$ of $U,$ Theorem \ref{sobolev_embedding_theorem} and Example \ref{hormander_example} imply that
    \[
        \|u\|_{L_{\frac{p}{1-p/Q}}(U)} \leq C_K(\|u\|_{L_p(U)}+\|X_1u\|_{L_p(U)}+\cdots+\|X_{n}u\|_{L_p(U)}).
    \]
    On the other hand, if $p>Q,$ we have
    \[
        \|u\|_{L_\infty(U)} \leq C_K(\|u\|_{L_p(U)}+\|X_1u\|_{L_p(U)}+\cdots+\|X_{n}u\|_{L_p(U)}).
    \]
    In this example $Q$ represents a worst-case-scenario and is in many cases much larger than necessary.

\end{example}

\section{Applications to elliptic regularity in $L_p$}\label{elliptic_regularity_section}
Recall that the Sobolev space $W^{s}_{p,\loc}(X;\Fc)$ associated to a filtered manifold $(X,\Fc)$ was given in Definition \ref{sobolev_spaces}. The following is an immediate consequence of Theorem \ref{main_theorem} and the composition properties of the AMY calculus.
\begin{lemma}\label{sobolev_mapping}
    If $s \in \Rl$ and $T \in \Psi^m_{\Fc}(X).$ If $1<p<\infty,$ then $T$ is a continuous mapping
    \[
        T:W^{s+\Re(m)}_{p,\loc}(X;\Fc)\to W^{s}_{p,\loc}(X;\Fc).
    \]
\end{lemma}

One application of Lemma \ref{sobolev_mapping} in the unfiltered case is the \emph{a priori} regularity theory for elliptic pseudodifferential operators, which can be proved via the fact that an elliptic operator of order $m$ has an inverse of order $-m$ modulo smoothing operators (a parametrix).

It was not proved in \cite{AMY2022arxiv} that a maximally hypoelliptic differential operator has a parametrix in the calculus, instead it was only proved for operators obeying the so-called strong $*$-Rockland condition \cite[Definition 3.28]{AMY2022arxiv}. We do not need to recall the definition of that property here, we only record the following application of the existence of a parametrix.
\begin{theorem}\label{elliptic_regularity}
    Let $1<p<\infty,$ let $D\in \Psi^m_{\Fc}(X)$ satisfy the strong $*$-Rockland condition. We have the elliptic regularity
    \[
        u\in \Dc'(X)\text{ and }Du\in W^{s}_{p,\loc}(X;\Fc)\Rightarrow u\in W^{s+m}_{p,\loc}(X;\Fc).
    \]
\end{theorem}
\begin{proof}
    It is proved in \cite[Theorem 3.30]{AMY2022arxiv} that pseudodifferential operators having the strong $*$-Rockland property have a parametrix in $\Psi^{-m}_{\Fc}(X),$ i.e. that there exists $E\in \Psi^{-m}_{\Fc}(X)$ such that $ED-1$ is smoothing. Since
    \[
        u = EDu+(1-ED)u.
    \]
    and $Du\in L_{p,\loc}(X),$
    Lemma \ref{sobolev_mapping} ensures that $EDu \in W^{s+m}_{p,\loc}(X),$ and since $(1-ED)u\in C^\infty(X)$ we conclude the proof.
\end{proof}

Applications of elliptic regularity are well-known. As an example, the following is an existence and \emph{a priori} regularity result for a certain non-linear PDE.
\begin{theorem}
    Let $(X,\Fc)$ be a compact filtered manifold, equipped with a Riemannian metric $g,$ and let $D\in \mathrm{DO}^{m}_{\Fc}(X)$ be a symmetric operator with respect to the inner product on $C^\infty(X)$ induced by $g.$ Assume that $D^2$ satisfies the strong $*$-Rockland condition.

    Let $p>1.$ For all real-valued $f \in L_{\frac{p+1}{p}}(X),$ there exists $u\in W^{2m}_{\frac{p+1}{p}}(X;\Fc)$ such that
    \[
        D^2u+u|u|^{p-1} = f.
    \]
\end{theorem}
\begin{proof}
    Let $u\in W^{m}_2(X;\Fc)\cap L_{p+1}(X)$ be real-valued, and define
    \[
        E(u) := \frac12\|Du\|_{L_2(X,g)}^2+\frac{1}{p+1}\|u\|_{L_{p+1}(X,g)}^{p+1}-\int_{X} f(x)u(x)\, dg(x).
    \]
    This functional is convex, coercive and lower bounded on the real-valued subspace of $W^{m}_2(X)\cap L_{p+1}(X),$ and hence has an infimum \cite[Theorem 1.5.6]{BadialeSerra2011}. By variational calculus the infimum $u$ satisfies
    \[
        D^2u+u|u|^{p-1} = f
    \]
    in the weak sense. Since $u\in L_{p+1}(X),$ it follows that $D^2u\in L_{\frac{p+1}{p}}(X).$
    By Theorem \ref{elliptic_regularity}, we deduce that $u\in W^{2m}_{\frac{p+1}{p}}(X;\Fc).$
\end{proof}

\end{document}